\documentclass[11pt]{amsart}

 \usepackage{amsmath,amsthm,amsfonts,amssymb,verbatim}
 \usepackage{url}
 \usepackage{graphicx}
 \usepackage[all]{xy}
 \usepackage{wrapfig}
 \usepackage{pinlabel}
 \usepackage{subfigure}
 
\def\co{\colon\thinspace}

\newcommand{\tD}{\mathbf{D}}
\newcommand{\tA}{\mathbf{A}}
\newcommand{\tT}{\mathbf{T}}
\newcommand{\tTz}{\mathbf{T_0}}
\newcommand{\tTo}{\mathbf{T_1}}
\newcommand{\tTzn}{\mathbf{T}^{\boxtimes n}_{\mathbf{0}}}
\newcommand{\tTon}{\mathbf{T}^{\boxtimes n}_{\mathbf{1}}}

\newcommand{\bu}{\bullet}

\newcommand{\sA}{\mathcal{A}}
\newcommand{\sB}{\mathcal{B}}

\newcommand{\sI}{\mathcal{I}}

\newcommand{\bZ}{\mathbb{Z}}
\newcommand{\bQ}{\mathbb{Q}}
\newcommand{\bR}{\mathbb{R}}

\newcommand{\bI}{\mathbb{I}}
\newcommand{\bF}{\mathbb{F}}

\newcommand{\pq}{\frac{p}{q}}

\newcommand{\HFhat}{\widehat{\operatorname{HF}}}

\newcommand{\rk}{\operatorname{rk}}

\newcommand{\CFhat}{\widehat{\operatorname{CF}}}

\newtheorem{theorem}{Theorem}

\newtheorem{definition}[theorem]{Definition}
\newtheorem{question}[theorem]{Question}
\newtheorem{corollary}[theorem]{Corollary}
\newtheorem{proposition}[theorem]{Proposition}
\newtheorem{remark}[theorem]{Remark}
\newtheorem{lemma}[theorem]{Lemma}
\newtheorem{conjecture}[theorem]{Conjecture}

\title[L-spaces and left-orderability]{On L-spaces and left-orderable fundamental groups}
\date{July 25, 2011}

\author[Steven Boyer]{Steven Boyer}
\thanks{Steven Boyer was partially supported by NSERC grant RGPIN 9446-2008}
\address{D\'epartement de Math\'ematiques, Universit\'e du Qu\'ebec \`a Montr\'eal, 201 avenue du Pr\'esident-Kennedy, Montr\'eal, QC H2X 3Y7.}
\email{boyer.steven@uqam.ca}
\urladdr{http://www.cirget.uqam.ca/boyer/boyer.html}

\author[Cameron McA. Gordon]{Cameron McA. Gordon}
\thanks{Cameron Gordon was partially supported by NSF grant DMS-0906276.}
\address{Department of Mathematics, University of Texas at Austin, 1 University Station, Austin, TX 78712.}
\email{gordon@math.utexas.edu}
\urladdr{http://www.ma.utexas.edu/text/webpages/gordon.html}

\author[Liam Watson]{Liam Watson}
\thanks{Liam Watson was partially supported by an NSERC postdoctoral fellowship}
\address{Department of Mathematics, UCLA, 520 Portola Plaza, Los Angeles, CA 90095.}
\email{lwatson@math.ucla.edu}
\urladdr{http://www.math.ucla.edu/~lwatson}

\begin{document}

\begin{abstract}
Examples suggest that there is a correspondence between L-spaces and $3$-manifolds whose fundamental groups cannot be left-ordered. In this paper we establish the equivalence of these conditions for several large classes of such manifolds. In particular, we prove that they are equivalent for any closed, connected, orientable, geometric $3$-manifold that is non-hyperbolic, a family which includes all closed, connected, orientable Seifert fibred spaces. We also show that they are equivalent for the $2$-fold branched covers of non-split alternating links. To do this we prove that the fundamental group of the $2$-fold branched cover of an alternating link is left-orderable if and only if it is a trivial link with two or more components. We also show that this places strong restrictions on the representations of the fundamental group of an alternating knot complement with values in $\hbox{Homeo}_+(S^1)$. 
\end{abstract}

\maketitle

\section{Introduction}

In this paper all $3$-manifolds will be assumed to be orientable. 

Heegaard Floer homology is a package of 3-manifold invariants introduced by Ozsv\'ath and Szab\'o \cite{OSz2004-properties,OSz2004-invariants}. There are various versions of this theory, however for our purposes it will suffice to consider the simplest of these: the {\em hat} version, denoted $\HFhat$.  

\begin{definition} A closed, connected $3$-manifold $Y$ is an L-space if it is a rational homology sphere with the property that $\rk\HFhat(Y)=\left|H_1(Y;\bZ)\right|$. \end{definition}

L-spaces form the class of manifolds with minimal Heegaard Floer homology and are of interest for various reasons. For instance, such manifolds do not admit co-orientable taut foliations \cite[Theorem 1.4]{OSz2004-genus}. It is natural to ask if there are characterizations of L-spaces which do not reference Heegaard Floer homology (cf. \cite[Question 11]{OSz2005-survey}). Examples of L-spaces include lens spaces as well as all connected sums of manifolds with elliptic geometry \cite[Proposition 2.3]{OSz2005-lens}. These examples also enjoy the property that their fundamental groups cannot be {\it left-ordered}:

\begin{definition} A non-trivial group $G$ is called left-orderable if there exists a strict total ordering $<$ on its elements such that $g<h$ implies $fg<fh$ for all elements $f,g,h\in G$. \end{definition}

While the trivial group obviously satisfies this criterion, in this paper we will adopt the convention that it is {\it not} left-orderable. 

The left-orderability  of $3$-manifold groups has been studied in work of Boyer, Rolfsen and Wiest \cite{BRW2005}. An argument of Howie and Short \cite[Lemma 2]{HS1985} shows that the fundamental group of  an irreducible $3$-manifold $M$ with positive first Betti number is locally indicable, hence left-orderable \cite{BH1972}. More generally, such a group is left-orderable if it admits an epimorphism to a left-orderable group \cite[Theorem 1.1(1)]{BRW2005}. 

The aim of this note is to establish a connection between L-spaces and the left-orderability of their fundamental groups. Given the results we obtain and those obtained elsewhere \cite{CLW2011,CW2010,CW2011,Peters2009}, we formalise a question which has received attention in the recent literature in the following conjecture.

\begin{conjecture} \label{main conjecture}
An irreducible rational homology $3$-sphere is an L-space if and only if its fundamental group is not left-orderable. 
\end{conjecture}

It has been asked by Ozsv\'ath and Szab\'o whether L-spaces can be characterized as those closed, connected $3$-manifolds admitting no co-orientable taut foliations. Thus, in the context of Conjecture \ref{main conjecture} it is interesting to consider the following open questions: Does the existence of a co-orientable taut foliation on an irreducible rational homology $3$-sphere imply the manifold has a left-orderable fundamental group? Are the two conditions equivalent? Calegari and Dunfield have shown that the existence of a co-orientable taut foliation on an irreducible atoroidal rational homology $3$-sphere $Y$ implies that $\pi_1(Y)$ has a left-orderable finite index subgroup \cite[Corollary 7.6]{CD2003}. Of course, an affirmative answer to Conjecture \ref{main conjecture}, combined with  \cite[Theorem 1.4]{OSz2004-genus}, would prove that the existence of a co-orientable taut foliation implies left-orderable fundamental group. 

Our first result verifies the conjecture in the case of Seifert fibred spaces. 

\begin{theorem}\label{thm:sf} 
Suppose $Y$ is a closed, connected, Seifert fibred 3-manifold. Then $Y$ is an L-space if and only if $\pi_1(Y)$ is not left-orderable. 
\end{theorem}

The proof of this theorem in the case where the base orbifold of $Y$ is orientable depends on results of Boyer, Rolfsen and Wiest \cite{BRW2005} and Lisca and Stipsicz \cite{LS2007-iii}. This case of  Theorem \ref{thm:sf} has been independently observed by Peters \cite{Peters2009}. There are, on the other hand, many Seifert fibred rational homology $3$-spheres with non-orientable base orbifolds, and it is shown in \cite{BRW2005} that such manifolds have non-left-orderable fundamental groups. Theorem \ref{thm:sf} therefore yields an interesting class of L-spaces that, to the best of our knowledge, has not received attention in the literature (see also \cite[Theorem 3.32]{Watson-PhD}). 

The set of torus semi-bundles provides another interesting family of $3$-manifolds in which to examine the relationship between L-spaces and non-left-orderable fundamental groups. Such manifolds are unions of two twisted $I$-bundles over the Klein bottle and are either Seifert fibred or admit a Sol geometry. Indeed, let $N$ be a twisted $I$-bundle over the Klein bottle and set $T = \partial N$. There are distinguished slopes $\phi_0, \phi_1$ on $T$ corresponding to the two Seifert structures supported by $N$. Here $\phi_0$ denotes the fibre slope of the structure with base orbifold a M\"{o}bius band and $\phi_1$ that with base orbifold $D^2(2,2)$. The general torus semi-bundle is homeomorphic to an identification space $W(f) = N \cup_f N$ where $f: T \to T$ is a homeomorphism. Further, $W(f)$ is 

\indent \hspace{.5cm} $\bullet$ a Seifert fibre space if and only if $f$ identifies $\phi_i$ with $\phi_j$ for some $i, j \in \{0, 1\}$. 
 
\indent \hspace{.5cm} $\bullet$ a Sol manifold if and only if $f$ does not identify any $\phi_i$ with any $\phi_j$ for $i, j \in \{0, 1\}$. 
 
\indent \hspace{.5cm} $\bullet$ a rational homology $3$-sphere if and only if $f$ does not identify $\phi_0$ with $\phi_0$. 

Thus the generic torus semi-bundle is a rational homology sphere and a Sol manifold. 

\begin{theorem}\label{thm:semi-fibre} 
Suppose that $W$ is a torus semi-bundle. Then the  following statements are equivalent:  

\vspace{-.2cm} $(a)$ $H_1(W; \mathbb Q) \cong 0$. 

\vspace{-.2cm} $(b)$ $\pi_1(W)$ is not left-orderable. 

\vspace{-.2cm} $(c)$ $W$ is an L-space. 
\end{theorem} 
A key step in the proof of this result requires a computation of the bordered Heegaard Floer homology \cite{LOT2009} of the twisted $I$-bundle over the Klein bottle. An immediate consequence of it verifies Conjecture \ref{main conjecture} for Sol manifolds. 

\begin{corollary}\label{crl:Sol}
Suppose that $Y$ is a closed, connected $3$-manifold with $Sol$  geometry. Then $Y$ is an L-space if and only if $\pi_1(Y)$ is not left-orderable. 
\qed
\end{corollary}

Theorem \ref{thm:sf} and Corollary \ref{crl:Sol} combine to give the following general statement:

\begin{theorem}
Suppose that $Y$ is a closed, connected, geometric, non-hyperbolic $3$-manifold. Then $Y$ is an L-space if and only if $\pi_1(Y)$ is not left-orderable. \qed\end{theorem}

Ozsv\'ath and Szab\'o determined a large family of L-spaces - the $2$-fold covers of $S^3$ branched over a non-split alternating link \cite[Proposition 3.3]{OSz2005-branch}. Conjecture \ref{main conjecture} can be established in this setting as well. We prove: 

\begin{theorem}\label{thm:alternating} The fundamental group of the $2$-fold branched cover of an alternating link $L$ is left-orderable if and only if $L$ is a trivial link with two or more components. In particular, the fundamental group of the $2$-fold branched cover of a non-split alternating link is not left-orderable.\end{theorem}
Note that generically, the $2$-fold branched cover of an alternating link $L$ is hyperbolic. 

Josh Greene \cite{Greene} has found an alternate proof of Theorem \ref{thm:alternating}. There is also relevant recent work of Ito on 2-fold branched covers \cite{Ito2011} and Levine and Lewallen  on {\em strong} L-spaces (manifolds $Y$ for which $\CFhat(Y)\cong\HFhat(Y)$)\cite{LL}. 

The results above relate L-spaces and manifolds with non-left-orderable fundamental groups. Next we consider examples of non-L-spaces with left-orderable fundamental groups. An interesting family of non-L-spaces has been constructed by Ozsv\'ath and Szab\'o - those manifolds obtained by non-trivial surgery on a hyperbolic alternating knot \cite[Theorem 1.5]{OSz2005-lens}. 

Recall that a {\it special alternating knot} is a knot which has an alternating diagram each of whose Seifert circles bounds a complementary region of the diagram. Equivalently, it is an alternating knot such that either each of the crossings in a reduced diagram for the knot is positive or each is negative. 

\begin{proposition}  \label{prop:alt} 
Let $K$ be a prime alternating knot in $S^3$. 

$(1)$  If $q \ne 0$ and $S^3_{p/q}(K)$ is Seifert fibred, then $\pi_1(S^3_{p/q}(K))$ is left-orderable.

$(2)$ If $K$ is not a special alternating knot, then $\pi_1(S^3_{1/q}(K))$ is left-orderable for all non-zero integers $q$.

$(3)$ If $K$ is a special alternating knot, then either all crossings in a reduced diagram for $K$ are positive and $\pi_1(S^3_{1/q}(K))$ is left-orderable for all positive integers $q$, or all crossings in the diagram are negative and $\pi_1(S^3_{1/q}(K))$ is left-orderable for all negative integers $q$.
\end{proposition} 

In the case of the figure eight knot we can say a little more. 

\begin{proposition} \label{prop:f8}
Let $K$ be the figure eight knot. If $-4 < r = \pq < 4$, then $\pi_1(S^3_r(K))$ is left-orderable.
\end{proposition}
Clay, Lidman and Watson have shown that the fundamental group of $\pm 4$-surgery on the figure eight knot is left-orderable \cite[\S 4]{CLW2011}. 

It is evident that a non-trivial subgroup of a left-orderable group is left-orderable. Here is a question which arises naturally from the ideas of this paper. 

\begin{question}
{\it Is a rational homology $3$-sphere finitely covered by an L-space necessarily an L-space?}
\end{question}

We point out that it follows from Theorem \ref{thm:semi-fibre} that there exists a class of examples of 2-fold covers that behave in this way: each $\bQ$-homology sphere with Sol geometry admits a 2-fold cover which is a $\bQ$-homology sphere with Sol geometry. 

The non-left-orderability of the fundamental group of the $2$-fold branched cover of a prime knot in the $3$-sphere has an interesting consequence for certain representations of the fundamental group of its complement. 

\begin{theorem} \label{thm:square trivial}
Let $K$ be a prime knot in the $3$-sphere and suppose that the fundamental group of its $2$-fold branched cyclic cover is not left-orderable. If $\rho: \pi_1(S^3 \setminus K) \to \hbox{Homeo}_+(S^1)$ is a homomorphism such that $\rho(\mu^2) = 1$ for some meridional class $\mu$ in $\pi_1(S^3 \setminus K)$, then the image of $\rho$ is either trivial or isomorphic to $\mathbb Z/2$.
\end{theorem} 

\begin{corollary} \label{cor:alt square-trivial}
Let $K$ be an alternating knot and $\rho: \pi_1(S^3 \setminus K) \to \hbox{Homeo}_+(S^1)$ a homomorphism. If $\rho(\mu^2) = 1$ for some meridional class $\mu$ in $\pi_1(S^3 \setminus K)$, then the image of $\rho$ is either trivial or isomorphic to $\mathbb Z/2$.   
\end{corollary} 

Alan Reid has pointed out the following consequence of this corollary.

\begin{corollary} \label{cor:trace-field}
Suppose that $K$ is an alternating knot and let $\mathcal{O}_K(2)$ denote the orbifold with underlying set $S^3$ and singular set $K$ with cone angle $\pi$. Suppose further that $\mathcal{O}_K(2)$ is hyperbolic. If the  trace field of $\pi_1(\mathcal{O}_K(2))$ has a real
embedding, then it must determine a $PSU(2)$-representation. In other words, the quaternion algebra 
associated to $\pi_1(\mathcal{O}_K(2))$ is ramified at that embedding.
\end{corollary}

\subsection*{Outline} The paper is organized as follows. Theorem \ref{thm:sf} is proven in \S \ref{sec:sf}. Generalities on torus semi-bundles are dealt with in \S \ref{sec: semi-bundle} followed by an outline of an inductive proof of Theorem \ref{thm:semi-fibre}. The base case of the induction is dealt with in \S \ref{sec: base case} and the inductive step in \S \ref{sec: inductive step}. Theorem \ref{thm:alternating} is proven in \S \ref{sec:alt} while Propositions \ref{prop:alt} and  \ref{prop:f8} are dealt with in \S \ref{sec:fig 8}. Finally, in \S \ref{sec: square trivial} we prove Theorem \ref{thm:square trivial} and Corollaries \ref{cor:alt square-trivial} and  \ref{cor:trace-field}.

\subsection*{Acknowledgements} The authors thank Adam Clay, Josh Greene, Tye Lidman and Ciprian Manolescu for their comments on and interest in this work, Alan Reid for mentioning Corollary \ref{cor:trace-field}, Michael Polyak for showing them the presentation described in Section 3.1, and J\'ozef Przytycki for pointing out Wada's paper \cite{Wada1992}. They also thank Adam Levine, Robert Lipshitz, Peter Ozsv\'ath and Dylan Thurston for patiently answering questions about bordered Heegaard Floer homology, which proved to be a key tool for establishing Corollary \ref{crl:Sol}. 

\section{A characterization of Seifert fibred L-spaces} \label{sec:sf}

\subsection{Preliminaries on L-spaces} \label{preliminaries} 
We recall an important construction which gives rise to infinite families of L-spaces (see \cite[Section 2]{OSz2005-lens}). 

\begin{definition} \label{def: triad} 
Let $M$ be a compact, connected 3-manifold with torus boundary. Given a basis $\{\alpha,\beta\} \subset H_1(\partial M)$ the triple $(\alpha,\beta,\alpha+\beta)$ will be referred to as a triad whenever \[\left|H_1(M(\alpha+\beta);\bZ)\right|=\left|H_1(M(\alpha);\bZ)\right|+\left|H_1(M(\beta);\bZ)\right|.\]\end{definition}
Note that our boundary orientation differs from that of Ozsv\'ath and Szab\'o, resulting in a sign discrepancy in the definition of a triad. 

\begin{proposition}[Proposition 2.1 of \cite{OSz2005-lens}] \label{sum L-space} 
If $M$ admits a triad $(\alpha,\beta,\alpha+\beta)$ with the property that $M(\alpha)$ and $M(\beta)$ are L-spaces, then $M(\alpha+\beta)$ is an L-space as well. 
\end{proposition}


It follows by induction that each of the manifolds $M(n\alpha+\beta)$ is an L-space for $n\ge0$. More generally, we include a short proof of the following well-known fact:
\begin{proposition}[See Example 1.10 of \cite{OSz2006-lectures}] \label{all L-spaces}
Suppose that $M$ admits a triad $(\alpha,\beta,\alpha+\beta)$ with the property that $M(\alpha)$ and $M(\beta)$ are L-spaces. Then for all coprime pairs $p, q \geq 0$, $M(p \alpha + q\beta)$ is an L-space.
\end{proposition}

\begin{proof} Let $p, q$ be a coprime pair with $p, q \geq 0$. Without loss of generality we can suppose $p, q \geq 1$. Choose integers $a_1 \geq 0$ and $a_2, \ldots , a_r \geq 1$ such that 
$$\textstyle \frac{p}{q} = [a_1, a_2, \ldots , a_r] = a_1 + \cfrac{1}{a_2 + \cdots \cfrac{1}{a_r}}$$ 
Let $r(\frac{p}{q}) \geq 1$ be the minimal $r$ such that $\frac{p}{q}$ admits such an expansion. When $r(\frac{p}{q}) = 1$, the proposition follows from the remark after Proposition \ref{sum L-space}. 

Assume next that $r = r(\frac{p}{q}) \geq 2$ and that $M(p' \alpha + q' \beta)$ is an L-space for all coprime pairs $p', q' \geq 1$ such that $r(\frac{p'}{q'}) < r$. Write $\frac{p}{q} = [a_1, a_2, \ldots , a_r]$ as above and note that as $[a_1, a_2, \ldots , a_{r-1}, 1] = [a_1, a_2, \ldots , a_{r-1} + 1]$ we have $a_r \geq 2$. Set 
$$\textstyle\frac{p_1}{q_1} = [a_1, a_2, \ldots , a_{r-1}]$$ 
and 
$$\textstyle\frac{p_2}{q_2} = [a_1, a_2, \ldots , a_r - 1]$$ 
It follows from the basic properties of the convergents of a continued fraction that $\{p_1\alpha + q_1 \beta, p_2\alpha + q_2 \beta\}$ is a basis of $H_1(\partial M)$ and $\frac{p}{q} = \frac{p_1 + p_2}{q_1 + q_2}$. The latter shows that $$\left|H_1(M(p\alpha + q \beta);\bZ)\right|=\left|H_1(M(p_1\alpha + q_1 \beta);\bZ)\right|+\left|H_1(M(p_2\alpha + q_2 \beta);\bZ)\right|.$$
(See the proof of \cite[Theorem 4.7]{Watson2008}.) This establishes that $(p_1\alpha + q_1 \beta, p_2\alpha + q_2 \beta, p \alpha + q \beta)$ is a triad. Our induction hypothesis implies that $M(p_1\alpha + q_1 \beta)$ is an L-space. We also have that $M(p\alpha + q \beta)$ is an L-space as long as $M(p_2 \alpha + q_2 \beta)$ is one. The latter will be true if $r(\frac{p_2}{q_2})  < r$, but this may not be the case. On the other hand if $r(\frac{p_2}{q_2}) = r$, we must have $a_r - 1 \geq 2$. Thus a second induction on $a_r$ is sufficient to complete the proof. 
\end{proof}

For example, all sufficiently large surgeries on a Berge knot  (the conjecturally complete list of knots in $S^3$ admitting lens space surgeries \cite{Berge}) yield L-spaces.

\subsection{Seifert fibred L-spaces.}
Our notation for Seifert fibred spaces follows that of Boyer, Rolfsen and Wiest \cite{BRW2005} and is consistent with that of Scott \cite{Scott1983}. Let $Y$ be Seifert fibred with base orbifold $\sB$, and write $\sB=B(a_1,a_2,\ldots,a_n)$ for some surface $B$ with cone points of order $a_i > 1$. If $Y$ is a rational homology sphere then $B$ is either $S^2$ or $P^2$. 

Lisca and Stipsicz have shown \cite[Theorem 1.1]{LS2007-iii} that when $B = S^2$, $Y$ is an L-space if and only if $Y$ does not admit a horizontal foliation while Boyer, Rolfsen and Wiest proved that these $Y$ do not admit a horizontal foliation if and only if $\pi_1(Y)$ is not left-orderable \cite[Theorem 1.3(b)]{BRW2005}. Thus Theorem \ref{thm:sf} holds when $B = S^2$. (See also Peters \cite{Peters2009}.) To complete its proof, we must consider the case $B = P^2$. 

\subsection{The proof of Theorem \ref{thm:sf} when $B = P^2$.}
Let $Y$ be Seifert fibred with base orbifold $\sB = P^2(a_1, a_2, \ldots, a_n)$. By \cite{BRW2005}, $\pi_1(Y)$ is not left-orderable. Since $Y$ is orientable but $P^2$ is not, $Y$ is a rational homology sphere. Hence we are reduced to establishing the following proposition:

\begin{proposition} \label{prp:rp2} Suppose $Y$ is a Seifert fibred space with base orbifold $\sB = P^2(a_1, a_2, \ldots, a_n)$ where $a_i \geq 1$ if $n = 1$ and $a_i \geq 2$ otherwise. Then $Y$ is an L-space.\end{proposition} 

\begin{proof}
First suppose that $\sB=P^2(a_1)$ where $a_1 \geq 1$. Then $Y$ is obtained by filling $N = K \tilde{\times} I$, the twisted $I$-bundle over the Klein bottle, along some slope $\alpha$. The Seifert structure on $Y$ restricts to a circle bundle structure on $N$ with base space the M\"obius band. If $\phi_0$ is the fibre of this bundle, then $N(\phi_0) \cong S^1 \times S^2$. Further, $\Delta(\alpha, \phi_0) = a_1$ by Heil \cite{Heil1974}.

There is another Seifert structure on $N$ with base orbifold $D^2(2,2)$ and fibre $\phi_1$ such that $\Delta(\phi_1, \phi_0) = 1$. Then $Y = N(\alpha)$ is either the L-space $N(\phi_1) \cong P^3 \# P^3$ or it admits a Seifert structure over $S^2(2,2,\Delta(\alpha,\phi_1))$ where $\Delta(\alpha,\phi_1) \geq 1$. (See \cite{Heil1974}.) Since $\alpha \ne \phi_0$,  $Y$ is elliptic in the latter case and therefore is also an L-space \cite[Proposition 2.3]{OSz2005-lens}. Thus the proposition holds when $Y$ has at most one exceptional fibre. 

Suppose inductively that any Seifert fibred manifold with base orbifold $P^2(a_1,\ldots,a_r)$ is an L-space whenever $1 \leq r \leq n$. Fix a Seifert fibred manifold over $P^2(a_1,\ldots,a_{n+1})$ and recall that by hypothesis, $a_i \geq 2$ for all $i$. 

Let $\phi_0$ be the exceptional fibre of $Y$ corresponding to the cone point of index $a_{n+1}$ and denote the exterior of $\phi_0$ in $Y$ by $N$. Then $N$ is a Seifert fibred manifold with base orbifold $B_0(a_1, a_2, \ldots , a_n)$ where $B_0$ is a M\"{o}bius band. The {\it rational longitude} of $N$ is the unique slope $\phi_0$ on $\partial N$ which represents a torsion element of $H_1(N; \mathbb Z)$. Since $N(\phi_0) \cong S^1 \times S^2$, $\lambda_N = \phi_0$ 

It is convenient to identify the (oriented) slopes on $\partial N$ with primitive elements of $H_1(\partial N; \mathbb Z)$. Choosing a dual class $\mu \in H_1(\partial N; \mathbb Z)$ for $\phi_0$ (i.e. a class such that $\mu\cdot\phi_0=1$), we obtain a basis for $H_1(\partial N; \mathbb Z)$. For any slope $\gamma \ne \pm \phi_0$, we can write $\gamma = \pm (p \mu + q \phi_0)$ where $p \geq 1$.  The Dehn filling $N(\gamma)$ is Seifert fibred with base orbifold $P^2(a_1,\ldots,a_n, \Delta(\phi_0, \gamma)) = P^2(a_1,\ldots,a_n, p)$. 
In particular, if $\alpha$ denotes the meridional slope of $\phi_0$, then $Y = N(\alpha)$ so that $\alpha =  \pm (a_{n+1}  \mu + q \phi_0)$. Note as well that for any $q \in \bZ$, our induction hypothesis implies that $N(\mu+q\phi_0)$ is an L-space.

By \cite[Lemma 2.1]{Watson2008}, there is a constant $D_N > 0$ depending only on $N$ such that for each slope $\gamma$ on $\partial N$, $|H_1(N(\gamma); \mathbb Z)| = D_N \Delta(\gamma, \phi_0)$. Then as $\mu \cdot \phi_0 = 1$, \begin{align*}
\left|H_1(N(2\mu+\phi_0);\bZ)\right| 
&= D_N\Delta(2\mu+\phi_0,\phi_0) \\
&= D_N\Delta(\mu,\phi_0)+D_N\Delta(\mu+\phi_0,\phi_0) \\
&= \left|H_1(N(\mu);\bZ)\right|+\left|H_1(N(\mu+\phi_0);\bZ)\right|.   
\end{align*} 
It follows that $(\mu,\mu+\phi_0,2\mu+\phi_0)$ is a triad of slopes on $\partial N$. Since $N(\mu)$ and $N(\mu+\phi_0)$ are L-spaces, Proposition \ref{all L-spaces} implies that $N((p+q)\mu+ q\phi_0)$ is an L-space for all coprime pairs $p, q \geq 0$. Now $\{\frac{p+q}{q} : p, q \geq 0 \hbox{ are coprime}\} = (\mathbb Q \cap [1, \infty)) \cup \{\frac10\}$ so that $N(\gamma)$ is an L-space for all slopes $\gamma$ in the sector of $H_1(\partial N; \mathbb R)$ bounded by the lines $\{ t \mu : t \in \mathbb R_+\}$ and $\{ t (\mu + \phi_0): t \in \mathbb R_+\}$. Since $\mu$ was chosen as an arbitrary dual class to $\phi_0$, given an integer $q$, $N(\gamma)$ is an L-space for all $\gamma \in \mathcal{S}_q$ where $\mathcal{S}_q$ is the set of slopes in the sector of $H_1(\partial N; \mathbb R)$ bounded by the lines $\{ t (\mu+q\phi_0) : t \in \mathbb R_+\}$ and $\{ t (\mu + (q+1)\phi_0): t \in \mathbb R_+\}$. Then as $\displaystyle{\bigcup_{q = -\infty}^\infty  \mathcal{S}_q}$ is the set of slopes on $\partial N$ other than $\phi_0$, the proposition has been proved. 
\end{proof}

\section{Torus semi-bundles} \label{sec: semi-bundle}  

Let $N$ be an oriented twisted $I$-bundle over the Klein bottle and give $T = \partial N$ the induced orientation. We remarked in the proof of Proposition \ref{prp:rp2} that there are two distinguished slopes $\phi_0, \phi_1$ on $T$ corresponding to the two Seifert structures supported by $N$. Here $\phi_0$ is the fibre slope of the structure with base orbifold a M\"{o}bius band while $\phi_1$ is the slope of the structure with base orbifold $D^2(2,2)$. It is well-known that $\phi_0$ and $\phi_1$ can be oriented so that $\phi_0 \cdot \phi_1 = 1$. Do this and observe that  $\{\phi_0, \phi_1\}$ is a basis for $H_1(T)$. We will identify the mapping class group of $T$ with $GL_2(\mathbb Z)$ using this basis - the mapping class of a homeomorphism $f$ corresponds to the matrix of $f_*: H_1(T; \mathbb Z) \to H_1(T; \mathbb Z)$ with respect to $\{\phi_0, \phi_1\}$. 

The first homology of $T$ maps to a subgroup of index two in $H_1(N; \mathbb Z)$. In fact, $H_1(N; \mathbb Z) \cong \mathbb Z \oplus \mathbb Z / 2$ where $\phi_0$ generates the second factor and $\phi_1$ represents twice a generator of the first. It follows that $\phi_0$ is the rational longitude of $N$ and that for any slope $\gamma$ on $T$ we have 
$$|H_1(N(\gamma); \mathbb Z)| = 4 \Delta(\gamma, \phi_0)$$
Further, it is well-known that a filling of $N$ with finite first homology is either $P^3 \# P^3$ or admits an elliptic geometry. Hence $N(\gamma)$ is an L-space if and only if $\gamma \ne \phi_0$. 

Let $f$ be a homeomorphism of $T = \partial N$ and suppose that $f_* = \left(\begin{smallmatrix} a  &   b \\ c   &  d \end{smallmatrix}\right)$ with respect to $\phi_0, \phi_1$. Then $W(f) = N \cup_f (-\hbox{det}(f_*)) N$ is an oriented torus semi-bundle and each such semi-bundle can be obtained this way. We claim that 
$$|H_1(W(f); \mathbb Z)| = 16 |c|$$ 
In fact, if $M_1, M_2$ are two rational homology solid tori and $W = M_1 \cup_f M_2$, it follows from the homology exact sequence of the pair $(W, M_1)$ that $|H_1(W; \mathbb Z)| = d_1d_2|T_1||T_2| \Delta(\lambda_1, \lambda_2)$ where $\lambda_j$ is the rational longitude of $M_j$, $d_j \geq 1$ is its order in $H_1(M_j; \mathbb Z)$, and $T_j$ is the torsion subgroup of $H_1(M_j; \mathbb Z)$. In our case, $d_1 = d_2 = |T_1| = |T_2| = 2$ and $\Delta(\lambda_1, \lambda_2) = \Delta(\phi_0, f_*(\phi_0)) = |c|$. Thus $|H_1(W(f); \mathbb Z)| = 16 |c|$ as claimed. It follows that a torus semi-bundle $W$ is a rational homology $3$-sphere if and only if $|c| \geq 1$.

\begin{remark} \label{non-zero} 
{\rm The fibre classes $\phi_0, \phi_1$ are preserved up to sign by each homeomorphism of $N$. In fact, under the identification of the mapping class group of $T$ with $GL_2(\mathbb Z)$ described above, the image of the mapping class group of $N$ in that of $T$ is the subgroup $\{\pm \left(\begin{smallmatrix} 1  & 0 \\ 0   &  1 \end{smallmatrix}\right), \pm \left(\begin{smallmatrix} 1  & 0 \\ 0   &  -1 \end{smallmatrix}\right)\}$. Hence given a torus semi-bundle $W(f)$, we can always assume $f_* =   \left(\begin{smallmatrix} a  &   b \\ c   &  d \end{smallmatrix}\right)$ where $c \geq 0$ and $\det(f_*)$ is an arbitrary element of $\{\pm 1\}$.}
\end{remark}

\begin{proof}[Proof of Theorem \ref{thm:semi-fibre}]
The implication (c) $\Rightarrow$ (a) of the theorem is immediate. The implication (a) $\Rightarrow$ (b) is a consequence of the proof of \cite[Proposition 9.1(1)]{BRW2005}. There it is shown that if a torus semi-bundle $W(f)$ has a left-orderable fundamental group, then $\phi_0$ must be identified with $\pm \phi_0$ by $f_*$. Equivalently, $f_* = \left(\begin{smallmatrix} a  &   b \\ 0   &  d \end{smallmatrix}\right)$ with respect to the basis $\phi_0, \phi_1$ of $H_1(T; \mathbb Z)$, and therefore $W(f)$ is not a rational homology $3$-sphere. 

To complete the proof of Theorem \ref{thm:semi-fibre}, we must show that the implication (b) $\Rightarrow$ (c) holds. To that end, let $W(f)$ be a torus semi-bundle whose fundamental group is not left-orderable and suppose that $f_* = \left(\begin{smallmatrix} a  &   b \\ c   &  d \end{smallmatrix}\right)$ with respect to $\phi_0, \phi_1$. Note that $c \ne 0$ as otherwise $W(f)$ would be irreducible and have a positive first Betti number, so its fundamental group would be left-orderable \cite[Theorem 1.1(1)]{BRW2005}. Thus $|c| \geq 1$. By Remark \ref{non-zero} we can suppose $c \geq 1$ and $\det(f_*) = -1$. We will proceed by induction on $c$. The initial case $c = 1$ is dealt with in the next section using a bordered Heegaard Floer homology argument. See Theorem \ref{thm:rk16}. The inductive step is handled in \S \ref{sec: inductive step} using a surgery argument based on the the triad condition of Proposition \ref{sum L-space}. See Proposition \ref{prop: inductive case}. 
\end{proof}

\section{The bordered invariants of the twisted $I$-bundle over the Klein bottle.}\label{sec: base case}

Heegaard Floer homology has been extended to manifolds with connected boundary by Lipshitz, Ozsv\'ath and Thurston \cite{LOT2009} (this approach subsumes knot Floer homology \cite{OSz2004-knot,Rasmussen-PhD} and was preceded, for the case of sutured manifolds, by work of Juh\'asz \cite{Juhasz2006}). In this context, the invariants take the form of certain modules (described below) over a unital differential (graded) algebra $\sA$. Denote by $\sI\subset\sA$ the subring of idempotents. Our focus is on the bordered invariants of the twisted $I$-bundle over the Klein bottle, and as such we restrict our attention to the case of manifolds with torus boundary. This simplifies some of the objects in question, and the relevant  setup in this case is summarized nicely in the work of Levine \cite{Levine2010}. As such we will adhere to the notation and conventions of  \cite[Section 2]{Levine2010} in the arguments and calculations  that follow. We work with $\bF=\bZ/2$ coefficients throughout. 

\begin{figure}[ht!]
\labellist 
	\pinlabel $A$ at 136 454
	\pinlabel {\reflectbox{\rotatebox{180}{$A$}}} at 136 331
	\pinlabel $B$ at 136 252
	\pinlabel {\reflectbox{\rotatebox{180}{$B$}}} at 136 129
	
\small
	\pinlabel $R_1$ at 27 220 
	\pinlabel $R_2$ at 27 240
	\pinlabel $R_3$ at 27 278
	\pinlabel $R_4$ at 175 278
	\pinlabel $R_5$ at 245 278
	\pinlabel $R_6$ at 245 188

	\pinlabel $a_1$ at -8 212
	\pinlabel $a_2$ at -8 228
	\pinlabel $a_3$ at -8 253
	\pinlabel $a_4$ at -8 355
	
	\pinlabel $\beta_1$ at 128 428
	\pinlabel $\beta_2$ at 128 156
	
	\pinlabel $\alpha$ at 293 278
	
	\pinlabel $\alpha_1^a$ at 235 218 
	\pinlabel $\alpha_2^a$ at 212 318 
	
	\pinlabel $x_1$ at 146 410
	\pinlabel $x_2$ at 146 381
	\pinlabel $x_3$ at 146 363
	
	\pinlabel $y_1$ at 146 233
	\pinlabel $y_2$ at 146 223
	\pinlabel $y_3$ at 146 212
	\pinlabel $y_4$ at 146 201
	\pinlabel $y_5$ at 146 183
	
	\pinlabel $\left\{\phantom{\begin{array}{c}a\end{array}}\right.$ at -15 218
	\pinlabel $\left\{\phantom{\begin{array}{c}a\\a\end{array}}\right.$ at -15 240
	\pinlabel $\left\{\phantom{\begin{array}{c}a\\a\\a\\a\\a\\a\\a\end{array}}\right.$ at -15 303
	
	\pinlabel $\rho_1$ at -34 218
	\pinlabel $\rho_2$ at -34 239
	\pinlabel $\rho_3$ at -34 303
	\endlabellist
\includegraphics[scale=0.9]{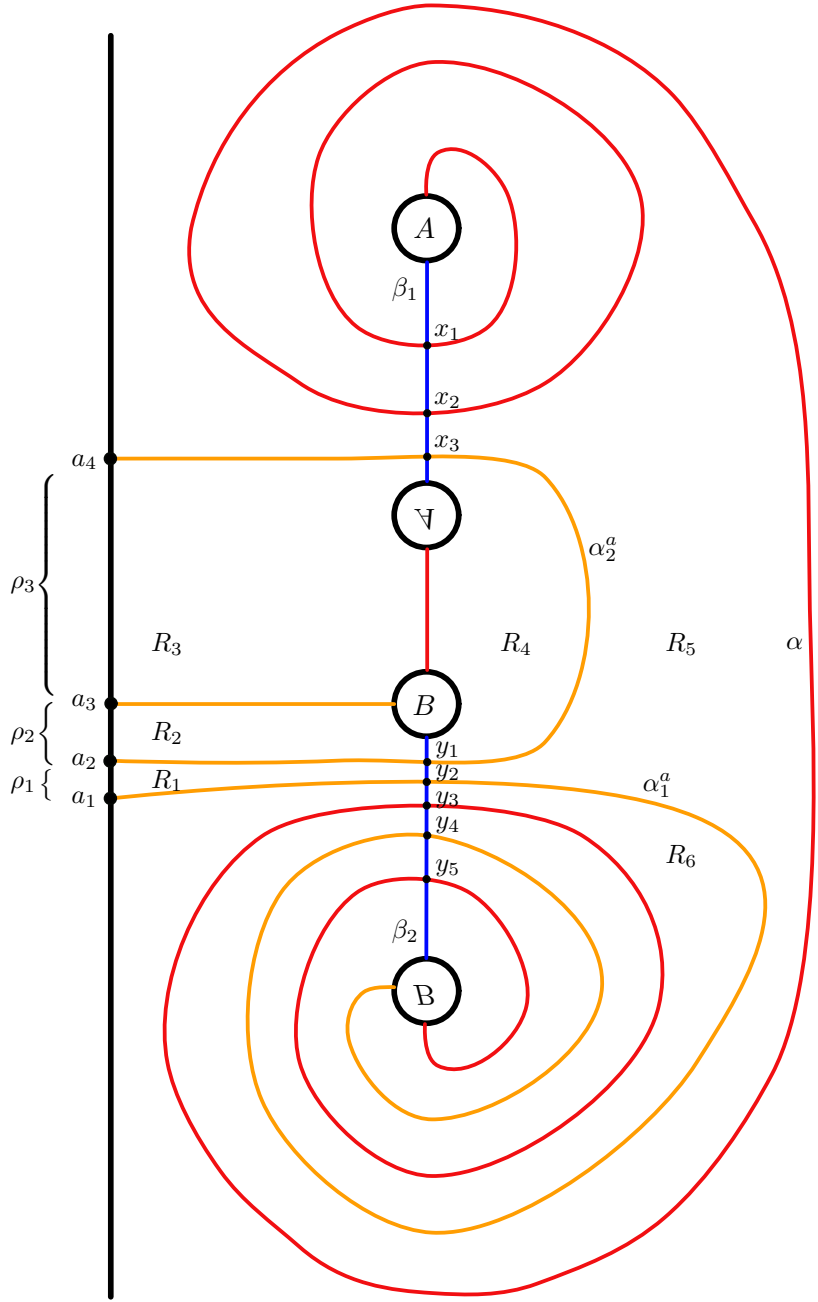}
\caption{A genus 2 bordered Heegaard diagram for the twisted $I$-bundle over the Klein bottle. The circles marked $A$ and $B$ are identified to give the handles, and the pair of twists in each handle encodes the Seifert structure with base orbifold $D^2(2,2)$.}\label{fig:twisted}
\end{figure}

\subsection{Determining the bordered invariants.} Recall that a (left) type D structure over $\sA$ is an $\bF$-vector space $V$ equipped with a left action of $\sI$ such that $V=\iota_0V\oplus\iota_1V$ and a map $\delta_1\co V\to\sA\otimes_\sI V$ (satisfying a compatibility condition, see \cite[Equation (3)]{Levine2010}, for example).

The twisted $I$-bundle over the Klein bottle $N$ is described by the bordered Heegaard diagram in Figure \ref{fig:twisted}. Recall that $N$ may be constructed by identifying a pair of solid tori along an essential annulus, giving rise to the Seifert structure with base orbifold $D^2(2,2)$. This identification of solid tori is reflected in the calculation of $\pi_1(N)=\langle a,b|a^2b^2\rangle$ via Seifert-van Kampen; alternatively, this presentation for $\pi_1(N)$ may be obtained using the single $\alpha$-curve in Figure \ref{fig:twisted} to obtain the relation. The reader may verify that the framing specified by this bordered diagram is consistent with $\{\phi_0,\phi_1\}$ by, for example, noting that the words $ab$ and $b^2$ (read from the arcs $\alpha_2^a$ and  $\alpha_1^a$, respectively) correspond to the peripheral elements $\phi_0$ and $\phi_1$, respectively.  

Let $\mathbf{D}=\widehat{\operatorname{CFD}}(N)$.  Our convention for the decomposition from the left-action of $\sI$ is to denote by $\circ$ those generators $x$ for which $\iota_0x=x$, and by $\bu$ those generators $y$ for which  $\iota_1y=y$.

\begin{proposition}\label{prp:typeD} The type D structure $\mathbf{D}$ is described by the directed graph
\[
\xymatrix@C=70pt@R=40pt{
{\bu}\ar@/^1pc/[d]|{\rho_{23}} & {\circ}\ar[r]|{\rho_1}\ar[d]|{\rho_3} & {\bu} \\
{\bu}\ar@/^1pc/[u]|{\rho_{23}} & {\bu}\ar[r]|{\rho_2} & {\circ}\ar[u]|{\rho_{123}}
}
\]
 \end{proposition}

The requisite map $\delta_1\co V\to\sA\otimes_\sI V$ for the type D structure is read from this directed graph as follows. The edge $\xymatrix@C=30pt{{\stackrel{x}{\circ}}\ar[r]|{\rho_1}&\stackrel{y}{\bu}}$, for example, indicates that there is a generator $x$ for which $\iota_0x=x$ and a generator $y$ for which $\iota_1y=y$ such that $\rho_1\otimes y$ appears as a summand in the expression for $\delta_1(x)$.  Since $\mathbf{D}$ decomposes into two summands, it will be convenient to write $\mathbf{D}=\mathbf{D_1}\oplus\mathbf{D_2}$ where the subscripts $\mathbf{1}$ an $\mathbf{2}$ denote the left and right connected components of the directed graph in Proposition \ref{prp:typeD}, respectively. 

\begin{proof}[Proof of Proposition \ref{prp:typeD}]
We determine $\tD$ directly from the definition, using the bordered Heegaard digram in Figure \ref{fig:twisted}. There are 8 generators for the underlying vector space $V$, partitioned according to  
\begin{align*}
\iota_0V &= \langle(x_1,y_1), (x_2,y_1), (x_3,y_3), (x_3,y_5) \rangle\\
\iota_1V &= \langle(x_1,y_2), (x_1,y_4), (x_2,y_2), (x_2,y_4)\rangle
\end{align*}
Adhering to the conventions for the case of a torus boundary outlined in \cite[Section 2]{Levine2010}, we begin by listing the possible domains in Table \ref{tab:domains}.  This follows from a case study, having observed that there are only 4 general types of domains $D$ depending on $n_1R_3+n_2R_4\subseteq D$ where the only possible multiplicities are  \[(n_1,n_2)\in\{(0,0),(0,1),(1,1),(1,2)\}\] (see \cite[Equation (4)]{Levine2010}, for example).

\begin{table}[ht!]
\begin{tabular}{|l|l|l|l|}
\hline
Domain & Source & Target & Notes and labels\\
\hline\hline
$R_1$ & $(x_1,y_1)$ & $(x_1,y_2)$ & Lost in cancelation. \\ 
& $(x_2,y_1)$ & $(x_2,y_2)$ & $D_1$\\
\hline\hline
$R_4$ & $(x_1,y_1)$ & $(x_3,y_5)$ & This is the only provincial domain.\\
\hline
$R_2+R_4$ & $(x_1,y_4)$ & $(x_3,y_3)$ &  $D_2$ \\
\hline\hline
$R_3+R_4$ & $(x_2,y_1)$ & $(x_1,y_4) $ & $D_3$ \\
\hline
 $R_2+R_3+R_4+R_6$ & $(x_2,y_4)$ & $(x_1,y_2)$ & $D_4$ \\
\hline
$R_2+R_3+R_4+R_5+R_6$ & $(x_3,y_3)$& $(x_2,y_1)$ & Inconsistent with idempotents. \\
\hline 
$R_1+R_2+R_3+R_4+R_5$ & $(x_3,y_5)$ & $(x_2,y_4)$ & Lost in cancelation. \\
\hline
$R_1+R_2+R_3+R_4+R_5+R_6$ & $(x_3,y_3)$ & $(x_2,y_2)$ & $D_5$ \\
\hline\hline
$R_2+R_3+2R_4$ & $(x_2,y_1)$& $(x_3,y_3)$& Inconsistent with idempotents. \\
\hline
$R_2+R_3+2R_4+R_5$ & $(x_1,y_2)$ & $(x_2,y_4)$&  $D_6$\\
\hline
$R_1+R_2+R_3+2R_4+R_5$ & $(x_1,y_1)$ & $(x_2,y_4)$& Lost in cancelation.\\
\hline
\end{tabular}\caption{A summary of possible domains for the bordered diagram in Figure \ref{fig:twisted}}
\label{tab:domains}\end{table}

The contributions of $D_1$ and $D_2$ to $\delta_1$ are immediate, as these regions are a bigon (containing the Reeb chord $-\rho_1$) and a rectangle (containing the Reeb chord $-\rho_2$), respectively. As a result we have operations
\[
\xymatrix@C=70pt@R=15pt{
 {\stackrel{(x_2,y_1)}{\circ}}\ar[r]|{\rho_1} & {\stackrel{(x_2,y_2)}{\bu}} \\
  {\stackrel{(x_1,y_4)}{\bu}}\ar[r]|{\rho_2} &  {\stackrel{(x_3,y_3)}{\circ}} 
}
\]
in $\tD$ (specifically, $\mathbf{D_2}$). Notice that potential contributions from the  domains $R_2+R_3+R_4+R_5+R_6$ and $R_2+R_3+2R_4$ (each of which would contribute a $\rho_{23}$) are ruled out since they are inconsistent with the idempotents: $(x_2,y_1),(x_3,y_3)\in\iota_0\mathbf{D_2}$. 

Further, the provincial domain $R_4$ is a rectangle and contributes $\delta_1(x_1,y_1)=(x_3,y_5)$. This operation may be eliminated via the edge reduction algorithm summarized in \cite[Section 2.6]{Levine2010}. Note that this eliminates 3 more domains from consideration (as in Table \ref{tab:domains}), since the (potential) contributions to $\delta_1$ from each of these is lost in the cancelation.

\begin{table}[ht!]
\begin{tabular}{|l|l|l|}
\hline
Domain & Contribution to $\delta_1$ & Sequence of Reeb chords \\
\hline\hline
$D_1$ & $\xymatrix@C=30pt{{(x_2,y_1)}\ar[r]|{\rho_1} & {(x_2,y_2)}} $ & $-\rho_1$ \\
\hline
$D_2$ & $\xymatrix@C=30pt{{(x_1,y_4)}\ar[r]|{\rho_2} & {(x_3,y_3)}}$& $-\rho_2$\\
\hline
$D_3$ & $\xymatrix@C=30pt{{(x_2,y_1)}\ar[r]|{\rho_3} & {(x_1,y_4)}} $ & $-\rho_3$ \\
\hline
 $D_4$ & $\xymatrix@C=30pt{{(x_2,y_4)}\ar[r]|{\rho_{23}} & {(x_1,y_2)}}$ & $(-\rho_2,-\rho_3)$ \\
\hline
$D_5$ & $\xymatrix@C=30pt{{(x_3,y_3)}\ar[r]|{\rho_{123}} & {(x_2,y_2)}}$ & $(-\rho_1,-\rho_2,-\rho_3)$ \\
\hline
$D_6$ & $\xymatrix@C=30pt{{(x_1,y_2)}\ar[r]|{\rho_{23}} & {(x_2,y_4)}}$&  $(-\rho_2,-\rho_3)$\\
\hline
\end{tabular}\caption{A summary of maps contributing to $\delta_1$ after canceling the differential corresponding to the provincial domain.}
\label{tab:maps}\end{table}

This leaves the summary of domains given in Table \ref{tab:maps}; to complete the proof we must justify the contribution from domain $D_i$ for $i=3,4,5,6$. From Table \ref{tab:maps}, we have
\[\mathbf{D_1} = \raisebox{28pt}{
\xymatrix@C=70pt@R=40pt{
{\bu}\ar@{..>}@/^1pc/[d]|{\rho_{23}}\\
{\bu}\ar @{..>}@/^1pc/[u]|{\rho_{23}} 
}}
\] and \[
\mathbf{D_2} = \raisebox{28pt}{
\xymatrix@C=70pt@R=40pt{
{\circ}\ar[r]|{\rho_1}\ar@{..>}[d]|{\rho_3} & {\bu} \\
{\bu}\ar[r]|{\rho_2} & {\circ}\ar @{..>}[u]|{\rho_{123}}
}}\]
where the dotted arrows denote the 4 contributions to $\delta_1$ that we have yet to justify. 

The domain $D_3$ is an annulus, containing the Reeb chord $-\rho_3$ in the boundary. Note that there is a single obtuse angle at $x_1$, and that cuts along either $\beta_1$ or $\alpha$ that start at $x_1$ decompose the annulus into a region with a single connected component (that is, such a cut meets the opposite boundary component). From this, following \cite[Pages 26--27]{Levine2010} for example, we may conclude that there  there are an odd number of holomorphic representatives for $D_3$, hence this permissible cut ensures that the domain supports a single holomorphic representative (counting modulo 2) of index 1,  establishing the contribution to $\delta_1$. 

The domain $D_4$ follows by similar considerations, having noted that there is a cut to the boundary starting at $y_4$ giving rise the the sequence $(-\rho_2,-\rho_3)$ in the boundary. The obtuse angle at $x_1$ ensures that this annular domain supports a single holomorphic representative (counting modulo 2) of index 1 as in the previous case. 

The domains $D_5$ and $D_6$ are more complicated, as each of these contains the region $R_5$ with multiplicity 1. However, we may treat each of these via considerations similar to those above, having first employed the following trick. 

The region $R_5$ may be simplified by altering the Heegaard diagram in Figure \ref{fig:twisted} by an isotopy: push the segment of $\beta_1$ between $x_2$ and $x_3$ to the right until a new bigon between $\alpha$ and $\beta_1$ is formed. Denote this new bigon by $R_7$ and its endpoints by $a$ and $b$ so that this new diagram produces $\delta_1(a,y_i)=(b,y_i)$ for $i=1,2,4$ in its corresponding type D structure. Note that the isotopy that removes $R_7$ realizes the homotopy achieved by canceling each of 
\[
\xymatrix@C=70pt@R=10pt{
 {\stackrel{(a,y_1)}{\circ}}\ar[r] & {\stackrel{(b,y_1)}{\circ}} \\
 {\stackrel{(a,y_2)}{\bu}}\ar[r] & {\stackrel{(b,y_2)}{\bu}} \\
  {\stackrel{(a,y_4)}{\bu}}\ar[r] & {\stackrel{(b,y_4)}{\bu}} 
 }
\]
via edge reduction. Denote by $R_5^+$ and $R_5^-$ the two new regions formed in the top and bottom of the diagram, respectively, having performed the isotopy producing $R_7$. These replace the region $R_5$; the remaining regions are unchanged.

We could work with this enlarged diagram directly, however it will be more convenient to simply identify the decompositions of $D_5$ and $D_6$ after the isotopy. In each case, the edge in question is replaced by a unique zig-zag.

The domain $D_5$ is decomposed into $D_5^+=R_2+R_3+R_4+R_5^++R_6$ and $D_5^-=R_1+R_5^-$. The latter is a rectangle containing the Reeb chord $-\rho_1$ in the boundary, while the former is a domain with the same structure as $D_4$ (containing the sequence $(-\rho_2,-\rho_3)$). This results in the zig-zag 
\[
\xymatrix@C=70pt@R=10pt{
 {\stackrel{(x_3,y_3)}{\circ}}\ar[r]|{\rho_1} & {\stackrel{(b,y_4)}{\bu}}& {\stackrel{(a,y_4)}{\bu}}\ar[r]|{\rho_{23}} \ar[l] & {\stackrel{(x_2,y_2)}{\bu}} 
 }
\]
which reduces to $\xymatrix@C=30pt{{\stackrel{(x_3,y_3)}{\circ}}\ar[r]|{\rho_{123}}&\stackrel{(x_2,y_2)}{\bu}}$ as claimed. 

The domain $D_6$ is decomposed into $D_6^+=R_3+R_4+R_5^+$ and $D_6^-=R_2+R_4+R_5^-$. Each of these is an annulus with the same structure as $D_3$, containing the Reeb chords $-\rho_3$ and $-\rho_2$, respectively. This results in the zig-zag  \[
\xymatrix@C=70pt@R=10pt{
 {\stackrel{(x_1,y_2)}{\bu}}\ar[r]|{\rho_2} & {\stackrel{(b,y_1)}{\circ}}& {\stackrel{(a,y_1)}{\circ}}\ar[r]|{\rho_{3}} \ar[l] & {\stackrel{(x_2,y_4)}{\bu}} 
 }
\]
which reduces to $\xymatrix@C=30pt{{\stackrel{(x_1,y_2)}{\bu}}\ar[r]|{\rho_{23}}&\stackrel{(x_2,y_4)}{\bu}}$ as claimed. 
\end{proof}

Recall that a (right) type A structure over $\sA$ is an $\bF$-vector space $V$ equipped with a right action of $\sI$ such that $V=V\iota_0\oplus V\iota_1$ and multiplication maps $m_{k+1}\co V\otimes_\sI\sA^{\otimes_\sI k}\to V$ satisfying the $\sA_\infty$ relations (see \cite[Equation (2)]{Levine2010}, for example). Let $\mathbf{A}=\widehat{\operatorname{CFA}}(N)$, where $\mathbf{A}=\mathbf{A_1}\oplus\mathbf{A_2}$ as above. For the present purposes, it suffices to formally define $\mathbf{A}=\widehat{\operatorname{CFAA}}(\bI,0)\boxtimes\tD$, where $\widehat{\operatorname{CFAA}}(\bI,0)$ is the identity type AA bimodule described in \cite[Section 10.1]{LOT2010} and summarized in  \cite[Section 2.4]{Levine2010}. Considering each summand of $\tD$ separately, we have $\mathbf{A_1}=\widehat{\operatorname{CFAA}}(\bI,0)\boxtimes\mathbf{D_1}$ and $\mathbf{A_2}=\widehat{\operatorname{CFAA}}(\bI,0)\boxtimes\mathbf{D_2}$. Note that this type A structure may be explicitly calculated from this information; by construction $\tA$ is bounded.

Our conventions and those of \cite[Section 10.1]{LOT2010} ensure that \[\CFhat(W(f)) \cong (\tA\boxtimes\tD,\partial^\boxtimes)\] where $W(f)=N\cup_f N$ and $f$ is given by $ \left(\begin{smallmatrix} 0  &   1 \\ 1   &  0 \end{smallmatrix}\right)$ (applying the main pairing result of Lipshitz, Ozsv\'ath and Thurston \cite[Theorem 1.3]{LOT2009}). One checks, for example, that $H_1(W(f);\bZ)\cong \bZ/4\oplus\bZ/4$ identifying the fact that $f(\phi_0)=\phi_1$ and $f(\phi_1)=\phi_0$ (the Heegaard diagram for the type A side is made explicit by appending the diagram of \cite[Figure 21]{LOT2010} to that of Figure \ref{fig:twisted}). As a vector space, $\tA\boxtimes\tD$ is generated by $ \tA\otimes_\sI\tD$. Recalling that $\delta_k=(\operatorname{id}_{\sA^{\otimes k-1}}\otimes\delta_1)\circ\delta_{k-1}$ with $\delta_0=\operatorname{id}_\tD$, the differential is defined by \[\partial^\boxtimes(x\otimes y)=\sum_{k=0}^\infty(m_{k+1}\otimes\operatorname{id}_\tD)(x\otimes\delta_k(y))\] (this is well-defined since $\tA$ is bounded). By a direct calculation one may show that $\rk \HFhat (W(f)) =16$. However, as $W(f)$ is a Seifert fibred $\bQ$-homology sphere, Theorem \ref{thm:sf} (together with the discussion in Section \ref{sec: semi-bundle}) ensures that this manifold is an L-space and the observation about the total rank of $\HFhat (W(f))$ follows immediately. 

\subsection{Changes of framing} Denote by $\tau_0$ and $\tau_1$ the Dehn twists along $\phi_0$ and $\phi_1$, respectively. In \cite[Section 10.2]{LOT2010}, the type DA bimodules corresponding to these mapping classes are described, and these may be composed via the box tensor product to change the framing on $\tD$. We note that our conventions are such that the Dehn twists $\tau_0$ and $\tau_1$ correspond to the Dehn twists $\tau_m$ and $\tau_l$, respectively, in \cite[Section 10.2]{LOT2010}. This can be seen, for example, by considering the effect on the peripheral elements in the new bordered Heegaard diagram obtained by adjoining each of the diagrams of \cite[Figure 25]{LOT2010} to the boundary of the diagram in Figure \ref{fig:twisted} (realizing the change of framing). 

Let $\tTz=\widehat{\operatorname{CFDA}}(\tau_0,0)$. As has been noted above, our conventions ensure that \[\CFhat(W(f)) \cong \tA\boxtimes\tD\] where $f$ is given by $\left(\begin{smallmatrix} 0  &   1 \\ 1   &  0 \end{smallmatrix}\right)$. As a result, $\tA\boxtimes\tTzn\boxtimes\tD$ gives a complex $\CFhat(W(f))$ for the manifold $W(f)=N\cup_f N$ where the homeomorphism $f$ is specified by $\left(\begin{smallmatrix} n  &   1 \\ 1   &  0 \end{smallmatrix}\right) = \left(\begin{smallmatrix} 1  &   n \\ 0   &  1 \end{smallmatrix}\right) \left(\begin{smallmatrix} 0  &   1 \\ 1   &  0 \end{smallmatrix}\right)$ for any integer $n$. We interpret negative values for $n$ via $\tT^{\boxtimes (-1)}_{\mathbf{0}}=\widehat{\operatorname{CFDA}}(\tau_0^{-1},0)$, so that for any integer $n$ we get a type D structure $\tTzn\boxtimes\tD$.

\begin{proposition}\label{prp:phi-zero} For all $n\in \bZ$ we have 
\[\tTzn\boxtimes\tD \cong \tD \] where $\cong$ denotes chain homotopy equivalence of type D structures.
\end{proposition}
\begin{proof} A description of $\tTz$ is given in \cite[Figure 27]{LOT2010} via a directed graph; we will record only the subgraph relevant to the present calculations.

First consider the summand $\tD_\mathbf{1}$. Since $\iota_0\tD_\mathbf{1}=0$, we need only consider generators of $\tTz\iota_1$ (and relevant operations relating them). In particular, we have that $\tTz\iota_1$ is described by the directed graph
\[
\xymatrix@C=70pt@R=40pt{
{\stackrel{q}{\bu}}\ar@(ur,dr)|{\rho_{23}\otimes\rho_{23}}
}
\] 
where the label $\rho_{23}\otimes\rho_{23}$ indicates that $m(q,\rho_{23})=\rho_{23}\otimes q$ ($q$ is the single generator for which $\iota_1q\iota_1=q$). From this it is immediate that $\tTz\boxtimes\tD_\mathbf{1} = \tD_\mathbf{1}$ hence $\tTzn\boxtimes\tD_\mathbf{1} = \tD_\mathbf{1}$.

Next consider the summand $\tD_\mathbf{2}$. Since neither elements $\rho_{12}$ nor $\rho_{23}$ appear in $\tD_\mathbf{2}$, operations involving $\rho_{12}$ or $\rho_{23}$ on the right will not be used in the box tensor product when calculating $\delta_1$ for the type D structure $\tTz\boxtimes\tD_\mathbf{2}$. As a result, the relevant operations of  $\tTz$ are described by the directed graph
\[
\xymatrix@C=70pt@R=40pt{
{\stackrel{p}{\circ}}\ar[rr]|{\rho_1\otimes\rho_1 + \rho_{123}\otimes\rho_{123}}\ar@/^1pc/[dr]|{\rho_3\otimes(\rho_3,\rho_2)}
&& {\stackrel{q}{\bu}}\ar@/^1pc/[dl]|{\rho_{23}\otimes\rho_2} \\
&{\stackrel{r}{\star}}\ar@/^1pc/[ul]|{\rho_2\otimes1}\ar@/^1pc/[ur]|{1\otimes\rho_3} & 
}
\] 
where, for example, the edge labelled $\rho_3\otimes(\rho_3,\rho_2)$ indicates that $m(p,\rho_3,\rho_2)$ contains a summand $\rho_3\otimes r$. The action of $\sI$ on the generators is given by $\iota_0p\iota_0=p$ (denoted by $\circ$), $\iota_1q\iota_1=q$ (denoted by $\bu$), and $\iota_1 r \iota_0 =r$ (denoted by $\star$). The notation $\star$ is intended to indicate the change from $\circ$ to $\bu$ when calculating the box tensor product with a type D structure. Note that, for any element marked $\circ$ (an element in the $\iota_0$-summand) in a type D structure we have that $\tTz\boxtimes -$ produces an edge $\xymatrix@C=30pt{{\bu}\ar[r]|{\rho_2}& \circ}$ in the new type D structure. Now the box tensor product 
\[
\xymatrix@C=70pt@R=40pt{
{\circ}\ar[rr]|{\rho_1\otimes\rho_1 + \rho_{123}\otimes\rho_{123}}\ar@/^1pc/[dr]|{\rho_3\otimes(\rho_3,\rho_2)}
&&{\bu}\ar@/^1pc/[dl]|{\rho_{23}\otimes\rho_2}\ar@{} [dr] |{\boxtimes} & {\circ}\ar[r]|{\rho_1}\ar[d]|{\rho_3}\ar@{-->}[dr]|{(\rho_3,\rho_2)} & {\bu} \\\
&{\star}\ar@/^1pc/[ul]|{\rho_2\otimes1}\ar@/^1pc/[ur]|{1\otimes\rho_3} & & {\bu}\ar[r]|{\rho_2} & {\circ}\ar[u]|{\rho_{123}}
}
\] 
gives the type D structure of interest;  the dashed arrow labelled $(\rho_3,\rho_2)$ indicates the operation $\delta_2(x) =\rho_3\otimes\rho_2\otimes y$ for generators $x,y\in\iota_0\tD_\mathbf{2}$. As a result $\tTz\boxtimes\tD_\mathbf{2}$ is described by the directed graph
\[\xymatrix@C=70pt@R=40pt{
{\circ}\ar[ddr]|{\rho_3}\ar[r]|{\rho_1} & \bu \\
\bu\ar[u]|{\rho_2}\ar[d] & \circ\ar[u]|{\rho_{123}} \\
\bu\ar[r]|{\rho_{23}} & \bu\ar[u]|{\rho_2}
}\]
By edge reduction, this is chain homotopy equivalent to $\tD_\mathbf{2}$ as claimed. 

Combining these two calculations gives $\tTz\boxtimes\tD\cong\tD$, and this may be iterated to obtain $\tTzn\boxtimes\tD\cong\tD$ for $n>0$. A similar calculation yields the case $n<0$.\end{proof}

By contrast, setting  $\tTo=\widehat{\operatorname{CFDA}}(\tau_1,0)$ we have that $\CFhat(W(f))\cong \tA\boxtimes\tTon\boxtimes\tD$ for the torus semi-bundle $W(f)$ arising from identification via any homeomorphism $f$ of the form $\left(\begin{smallmatrix} 0  &   1 \\ 1   &  n \end{smallmatrix}\right)=\left(\begin{smallmatrix} 1  &   0 \\ n   &  1 \end{smallmatrix}\right)\left(\begin{smallmatrix} 0  &   1 \\ 1   &  0 \end{smallmatrix}\right)$. Again, we denote $\tT^{\boxtimes (-1)}_{\mathbf{1}}=\widehat{\operatorname{CFDA}}(\tau_1^{-1},0)$. In this setting $\tTon\boxtimes-$ is non-trivial (in the sense of Proposition \ref{prp:phi-zero}) when applied to $\tD$. The structure of $\tTon\boxtimes\tD$ is well-behaved, and may be easily computed (proceeding as in the argument above) for any $n\in \bZ$. However, this will not be needed (explicitly) in the present setting, so we leave it as an exercise for the interested reader. 

\subsection{An infinite family of L-spaces of rank 16: the base case} 
For the homeomorphism  $g$ described by 
$\left(\begin{smallmatrix} 1  &   a\\ 0   &  1 \end{smallmatrix}\right)\left(\begin{smallmatrix} 1  &   0 \\ b   &  1 \end{smallmatrix}\right)\left(\begin{smallmatrix} 0  &   1 \\ 1   &  0 \end{smallmatrix}\right)=\left(\begin{smallmatrix} a  &   ab+1 \\ 1   &  b \end{smallmatrix}\right)$ 
we have 
\[\CFhat(W(g))\cong \tA\boxtimes\tT^{\boxtimes b}_{\mathbf{1}}\boxtimes \tT^{\boxtimes a}_{\mathbf{0}}\boxtimes \tD\] with $W(g)=N\cup_g N$, for all $a,b\in\bZ$. This family of  torus semi-bundles is of immediate interest.

\begin{theorem}\label{thm:rk16} Let $N$ be the twisted $I$-bundle over the Klein bottle. Let $g$ be the homeomorphism $\partial N\to \partial N$ defined by the matrix $\left(\begin{smallmatrix} a  &   ab+1 \\ 1   &  b \end{smallmatrix}\right)$ for any $a,b\in\bZ$. Then the torus semi-bundle $W(g)=N\cup_g N$ is an L-space, with $\rk\HFhat(W(g))=16$.\end{theorem}

Note that when either $a=0$ or $b=0$, the resulting manifold is a Seifert fibered space. In this special case Theorem \ref{thm:rk16} follows from work of Boyer, Rolfsen and Wiest \cite{BRW2005} (see the discussion in Section \ref{sec: semi-bundle}) and Theorem \ref{thm:sf}. Generically however, $W(g)=N\cup_g N$ is a Sol manifold, and it is this case that is of present interest. 

\begin{proof}[Proof of Theorem \ref{thm:rk16}]
It suffices to prove that $\rk H_*(\tA\boxtimes\tT^{\boxtimes b}_{\mathbf{1}}\boxtimes \tT^{\boxtimes a}_{\mathbf{0}}\boxtimes \tD) = 16$. By Proposition \ref{prp:phi-zero}, 
\[\tT^{\boxtimes b}_{\mathbf{1}}\boxtimes \tT^{\boxtimes a}_{\mathbf{0}}\boxtimes \tD\cong\tT^{\boxtimes b}_{\mathbf{1}}\boxtimes \tD\] as type D structures, so that \[
H_*(\tA\boxtimes\tT^{\boxtimes b}_{\mathbf{1}}\boxtimes \tT^{\boxtimes a}_{\mathbf{0}}\boxtimes \tD)\cong H_*(\tA\boxtimes\tT^{\boxtimes b}_{\mathbf{1}}\boxtimes \tD)\]
Recall that $H_*(\tA\boxtimes\tT^{\boxtimes b}_{\mathbf{1}}\boxtimes \tD)\cong \HFhat(W(g))$ where $g$ is the homeomorphism defined by $\left(\begin{smallmatrix} 0  &   1 \\ 1   &  b \end{smallmatrix}\right)$. As previously observed, this is a Seifert fibred L-space with $|H_1(W(g);\bZ)|=16$, hence $\rk\HFhat(W(g))=16$ as claimed. 
\end{proof}

The key feature of this argument is the (more general) observation that post-composing any homeomorphim by $\tau_0^n$ gives $\left(\begin{smallmatrix} 1  &   n \\ 0   &  1 \end{smallmatrix}\right) \left(\begin{smallmatrix} a  &   b \\ c   &  d \end{smallmatrix}\right)=\left(\begin{smallmatrix} a+nc  &   b+nd \\ c   &  d \end{smallmatrix}\right)$; Proposition \ref{prp:phi-zero} implies that the (ungraded) Heegaard Floer homology is identical for the family of torus semi-bundles obtained via these homeomorphisms. The proof of Theorem \ref{thm:rk16} makes use of the fact that for $c=1$, choosing $n=-a$ yields a Seifert fibred torus semi-bundle.  


\section{The inductive step in the proof of Theorem \ref{thm:semi-fibre}} \label{sec: inductive step}
 
In this section we complete the proof of Theorem \ref{thm:semi-fibre} with the following proposition (cf. Theorem \ref{thm:rk16}).

\begin{proposition} \label{prop: inductive case}
Suppose that $W(f)$ is an L-space whenever $f_* = \left(\begin{smallmatrix} a  &   b \\ c   &  d \end{smallmatrix}\right)$ where $c = 1$ and $\det(f_*) = -1$. Then $W(f)$ is an L-space whenever $c \ne 0$.
\end{proposition}

The proof is at the end of this section. To set it up, let $f$ be a homeomorphism of $T$ with matrix $f_* = \left(\begin{smallmatrix} a  &   b \\ c   &  d \end{smallmatrix}\right)$ with respect to $\phi_0, \phi_1$. Our first goal is to understand conditions under which Dehn surgery on $W(f)$ along a knot contained in $T = \partial N$ yields an L-space. 

Let $\gamma$ be a slope on $T$ represented by a simple closed curve $K_\gamma \subset T \subset W(f)$. As $T$ is oriented, we have a homeomorphism $D_{K_\gamma}: T \to T$, well-defined up to isotopy, given by a Dehn twist along $K_\gamma$. On the level of homology 
$$(D_{K_\gamma})_*(\beta) = \beta + (\beta \cdot \gamma) \gamma \hbox{ for all } \beta \in H_1(T)$$

Denote the exterior of $K_\gamma$ in $W(f)$ by $M_{f, \gamma}$ and set $T_\gamma = \partial M_{f, \gamma}$. There is a basis $\{\mu_\gamma, \lambda_T\}$ of $H_1(T_\gamma)$ where $\mu_\gamma$ is a meridian of $K_\gamma$ and $\lambda_T$ is represented by a parallel of $K_\gamma$ lying on $T$. Orient $\lambda_T$ and $\mu_\gamma$ so that $\mu_\gamma \cdot \lambda_T = 1$ with respect to the induced orientation on $T_\gamma$. Our first goal is to  determine the constant $\epsilon(f, \gamma)$ such that $(\mu_\gamma, \epsilon(f, \gamma) \lambda_T, \mu_\gamma + \epsilon(f, \gamma) \lambda_T)$ is a triad (cf. Definition \ref{def: triad}). 

Note that 
$$M_{f, \gamma}(\mu_\gamma) = W(f)$$ 
while 
$$M_{f, \gamma}(\lambda_T) = N(\gamma) \# N(f_*(\gamma))$$ 
The latter is an $L$-space as long as neither $\gamma$ nor $f_*(\gamma)$ is $\pm \phi_0$. More precisely, let 
$$\gamma = r \phi_0 + s \phi_1 \;\;\; \hbox{ and } \;\;\; f_*(\gamma) = u \phi_0 + v \phi_1$$ 
Then 
$$u = ar+ bs \;\;\; \hbox{ and } \;\;\; v = cr + ds$$
The only filling of $N$ which is not an $L$-space is the $\phi_0$-filling. Further $|H_1(N(\gamma))| = 4 \Delta(\gamma, \phi_0) = 4|s| \;\;\; \hbox{ and } \;\;\; |H_1(N(f_*(\gamma)))| = 4|v|$. Thus $M_{f, \gamma}(\lambda_T)$ is an $L$-space if and only if $s, v \ne 0$. In this case, 
$$|H_1(M_{f, \gamma}(\lambda_T))| = 16|sv|$$
For $c, s, v \ne 0$ set 
$$\epsilon(f, \gamma) = -\hbox{sign}(csv)$$
and note that 
$|H_1(M_{f, \gamma}(\mu_\gamma + \epsilon(f, \gamma) \lambda_T) )| = |H_1(M_{f, \gamma}(\mu_\gamma))| + |H_1(M_{f, \gamma}(\lambda_T) )|$. Consequently $(\mu_\gamma, \epsilon(f, \gamma) \lambda_T, \mu_\gamma + \epsilon(f, \gamma) \lambda_T)$ 
is a triad.

It is well-known that $M_{f, \gamma}(\mu_\gamma + n \lambda_T) = W(g_n)$ where $g_n = f \circ D_{K_\gamma}^{-n}$. In particular, 
$(g_n)_*(\beta) = f_*(\beta) - n (\beta \cdot \gamma) f_*(\gamma)$. Hence
$(g_n)_* = \begin{pmatrix} a - nsu &   b + nru \\ c - nsv  &  d +nrv \end{pmatrix}$ and 
therefore 

\begin{lemma} \label{L-space} 
Suppose that $W(f)$ is an L-space and $\gamma, f_*(\gamma) \ne \pm \phi_0$. Then $W(g_n)$ is an $L$-space for all $n \geq 0$ where $g_n = f \circ D_{K_\gamma}^{-\epsilon(f, \gamma) n}$. Further, if $f_* = \left(\begin{smallmatrix} a  &   b \\ c   &  d \end{smallmatrix}\right)$ then \\ 
$$(g_n)_* = \begin{pmatrix} a - n\epsilon(f, \gamma)su &   b + n\epsilon(f, \gamma)ru \\ c - n\epsilon(f, \gamma)sv  &  d + n\epsilon(f, \gamma)rv \end{pmatrix}$$  
\qed 
\end{lemma}

\begin{lemma} \label{all columns-rows} 
Let $f$ and $\gamma$ be as above and let $n \geq 0$. Suppose that $W(g)$ is an L-space whenever $g_*$ has first column $\left(\begin{smallmatrix} a \\ c  \end{smallmatrix}\right)$. Then $W(h)$ is an L-space whenever $h_*$ has first column $\left(\begin{smallmatrix} a - n\epsilon(f, \gamma)su \\ c - n\epsilon(f, \gamma)sv  \end{smallmatrix}\right)$.  
\end{lemma}

\begin{proof}
Consider a homeomorphism $h$ of $T$ such that $h_*$ has first column $\left(\begin{smallmatrix} a - n\epsilon(f, \gamma)su \\ c - n\epsilon(f, \gamma)sv  \end{smallmatrix}\right)$. To see that $W(h)$ is an L-space we can suppose that $\det(h_*) = \det(f_*)$ by Remark \ref{non-zero}. Then Lemma \ref{L-space} implies that the second column of $h_*$ can then be written
$$\begin{pmatrix} b' \\ d' \end{pmatrix} = \begin{pmatrix} b + n\epsilon(f, \gamma)ru \\ d + n\epsilon(f, \gamma)rv \end{pmatrix} + m \begin{pmatrix} a - n\epsilon(f, \gamma)su \\ c - n\epsilon(f, \gamma)sv  \end{pmatrix} = \begin{pmatrix} b + ma \\ d + mc \end{pmatrix}  + n \epsilon(f, \gamma) (r - ms) \begin{pmatrix} u \\     v \end{pmatrix}$$ 
for some $m \in \mathbb Z$. 

By hypothesis, $W(k)$ is an L-space where $k$ is the homeomorphism of $T$ with matrix 
$k_* = \left(\begin{smallmatrix} a  &   b + ma \\ c   &  d + mc \end{smallmatrix}\right)$. Let $\gamma' = r' \phi_0 + s \phi_1$ where $r' = r - ms$. Then 
$k_*(\gamma') = (a(r  - ms) + (b + ma)s) \phi_0 + (c(r - ms) + (d + mc)s) \phi_1 = (ar + bs) \phi_0 + (cr + ds)\phi_1 = u \phi_0 + v \phi_1 = f_*(\gamma)$. Hence $\gamma', k_*(\gamma') \ne \pm \phi_0$. Further note that $\epsilon(k, \gamma') = -\hbox{sign}(csv) = \epsilon(f, \gamma)$. Lemma \ref{L-space} then shows that $W(k_n)$ is an L-space where 
$$(k_n)_* = \begin{pmatrix} a - n\epsilon(f, \gamma)su &   b + ma + n\epsilon(f, \gamma)r'u \\ c - n\epsilon(f, \gamma)sv  &  d + mc + n\epsilon(f, \gamma)r'v \end{pmatrix}$$ 
The second column of $(k_n)_*$ is  
$$\begin{pmatrix} b + ma + n\epsilon(f, \gamma)r'u \\ d + mc + n\epsilon(f, \gamma)r'v \end{pmatrix} = \begin{pmatrix} b + ma \\ d + mc \end{pmatrix}  + n \epsilon(f, \gamma) (r - ms) \begin{pmatrix} u \\     v \end{pmatrix} = \begin{pmatrix} b' \\ d' \end{pmatrix}$$ 
Hence $(k_n)_* = h_*$, which completes the proof.
\end{proof}

\begin{proof}[Proof of Proposition \ref{prop: inductive case}]
Consider a torus semi-bundle $W(f)$ where $f_* = \left(\begin{smallmatrix} a  &   b \\ c   &  d \end{smallmatrix}\right)$ and $c \ne 0$. We will show that $W(f)$ is an L-space assuming that this is the case when $c = 1$ and $\det(f_*) = -1$. By Remark \ref{non-zero}, this assumption implies that $W(f)$ is an L-space whenever $|c| = 1$.  We proceed by induction on $|c|$.

Let $C$ be an integer of absolute value $2$ or larger and suppose that $W(f)$ is an L-space whenever $f_* = \left(\begin{smallmatrix} a  &   b \\ c   &  d \end{smallmatrix}\right)$ where $1 \leq |c| < |C|$. Consider a torus semi-bundle $W(F)$ where $F_* = \left(\begin{smallmatrix} A  &   B \\ C   &  D \end{smallmatrix}\right)$. By Remark \ref{non-zero} we can suppose that $C > 1$. Choose integers $a, c$ such that $aC - cA = 1$ and $0 < c < C$. Set $b = A - 2a, d = C - 2c$. Then $ad - bc = aC - cA = 1$ so $f_* = \left(\begin{smallmatrix} a  &   b \\ c   &  d \end{smallmatrix}\right) \in SL_2(\mathbb Z)$. By induction,  $W(g)$ is an L-space for all $g$ such that $g_*$ has first column $\left(\begin{smallmatrix} a  \\ c   \end{smallmatrix}\right)$.

Take $\gamma = \phi_0 + \phi_1$ so that $f_*(\gamma) = (a+b) \phi_0 + (c+d) \phi_1$. Then in the notation established earlier in this section, $u = a+b$ and $v = c+d  = C - c > 0$. Hence $\epsilon(f, \gamma) = -\hbox{sign}(csv) = -1$ so that Lemma \ref{all columns-rows} implies that $W(G)$ is an L-space for all $G$ such that $G_*$ has first column $\left(\begin{smallmatrix} a - \epsilon(f, \gamma)su \\ c -  \epsilon(f, \gamma)sv  \end{smallmatrix}\right) = \left(\begin{smallmatrix} a + (a+b) \\ c + (c+d)  \end{smallmatrix}\right) = \left(\begin{smallmatrix} A \\ C  \end{smallmatrix}\right)$. In particular $W(F)$ is an L-space. This completes the induction.
\end{proof}

\section{$2$-fold branched covers of alternating links} \label{sec:alt}

In this section we prove Theorem \ref{thm:alternating}. 

\subsection{Wada's group} 

Let $L$ be a link in $S^3$ and $D$ a diagram for $L$. Label the arcs of the diagram $1$ through $n$ as in Figure \ref{fig:arc-labels}. Define a group $\pi(D)$ as follows: $\pi(D)$ has generators $a_1, a_2, \ldots, a_n$ in one-one correspondence with the arcs of $D$, and relations of the form 
$$a_k a_j^{-1} a_i a_j^{-1}. \eqno{(1)}$$ 
in one-one correspondence with the crossings of $D$. 
Note that this relation is well-defined, as it is invariant under interchanging the indices $i$ and $k$.  

\begin{figure}[ht!]
\begin{center}
\labellist \small
	\pinlabel $j$ at 314 475
	\pinlabel $j$ at 204 335
	\pinlabel $i$ at 204 475
	\pinlabel $k$ at 314 335
\endlabellist
\raisebox{0pt}{\includegraphics[scale=0.5]{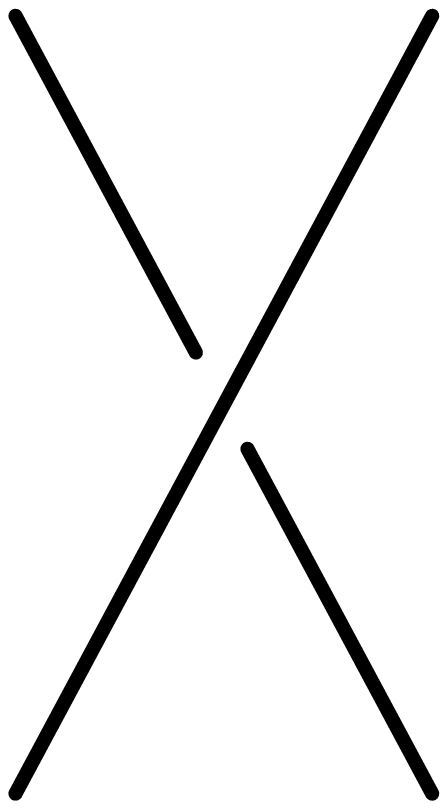}}
\qquad\qquad\qquad
\raisebox{0pt}{\includegraphics[scale=0.5]{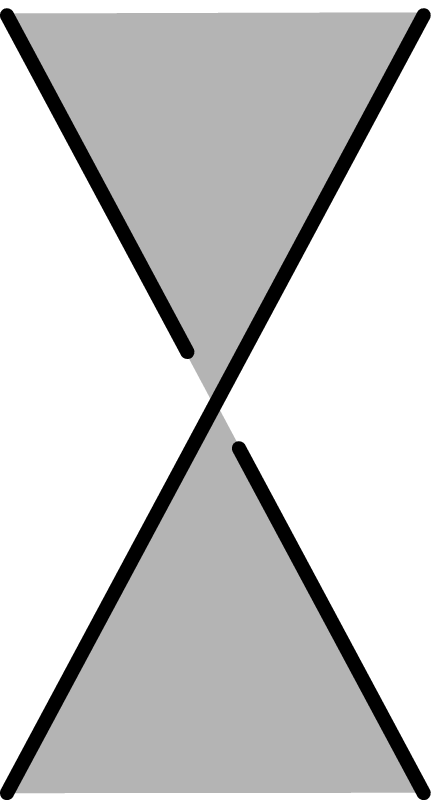}}
\end{center}
\caption{Labeling the arcs at a crossing (left), and the checkerboard convention for the black graph (right).}
\label{fig:arc-labels}
\end{figure}

This presentation was considered by Wada \cite{Wada1992}, who proved the following theorem. (See also \cite{Przytycki1998}.) We include a proof for completeness.

\begin{theorem} \label{polyak's group} 
$\pi(D) \cong \pi_1(\Sigma(L)) * \mathbb Z$ where $\Sigma(L)$ is the $2$-fold branched cover of $L$. 
\end{theorem}

\begin{proof} Let $M_L$ be the complement of $L$. The Wirtinger presentation of $\pi_1(M_L)$ corresponding to the diagram $D$ has generators $a_1, a_2, \ldots , a_n$ as above and a relation at each crossing, of the form 
$$a_k^{-1} a_j a_i a_j^{-1} \eqno{(2)}$$ 
for suitable choice of labels $i$ and $k$. 
Let $X$ be the $2$-complex of this presentation, with one $0$-cell $x$, $n$ oriented $1$-cells, and $n$ $2$-cells. Thus $\pi_1(X) \cong \pi_1(M_L)$. Let $\widetilde X$ be the (connected) double cover of $X$ determined by the homomorphism $\pi_1(X) \to \mathbb Z/2, a_i \mapsto 1$ for all $i$. In $\widetilde X$ the $0$-cell $x$ of $X$ lifts to two $0$-cells $x_1$ and $x_2$, and each $1$-cell $a_i$, say, of $X$ lifts to two $1$-cells that we will denote by $\alpha_i$ and $\beta_i$, where $\alpha_i$ is oriented from $x_1$ to $x_2$ and $\beta_i$ from $x_2$ to $x_1$. Let $\widetilde X_+$ be obtained from $\widetilde X$ by adjoining an arc $e$, identifying $\partial e$ with $\{x_1, x_2\}$. Then $\pi_1(\widetilde X_+) \cong \pi_1(\widetilde X) * \mathbb Z$. Taking as ``base-point" for $\widetilde X_+$ the maximal tree $e$ we obtain the following presentation for $\pi_1(\widetilde X_+)$: generators $\alpha_1, \beta_1, \alpha_2, \beta_2, \ldots , \alpha_n, \beta_n$ and pairs of relations, corresponding to the two lifts of each $2$-cell in $X$,
$$\beta_k^{-1} \beta_j \alpha_i \alpha_j^{-1} \eqno{(3)}$$
$$\alpha_k^{-1} \alpha_j \beta_i \beta_j^{-1}  \eqno{(4)}$$
Since $\pi_1(\widetilde X) \cong \pi_1(\widetilde M_L)$, where $\widetilde M_L$ is the double cover of $M_L$, and since $a_1, a_2, \ldots, a_n$ are meridians of $L$, we obtain a presentation for $\pi_1(\Sigma(L)) * \mathbb Z$ by adding the branching relations 
$$\alpha_1 \beta_1 = \alpha_2 \beta_2 = \ldots = \alpha_n \beta_n = 1$$
Thus eliminating $\beta_1 = \alpha_1^{-1}, \beta_2 = \alpha_2^{-1}, \ldots , \beta_n = \alpha_n^{-1}$, equations (3) and (4) become
$$\alpha_k \alpha_j^{-1} \alpha_i \alpha_j^{-1}$$
$$\alpha_k^{-1} \alpha_j \alpha_i^{-1} \alpha_j $$
Since the second relation is a consequence of the first, those relations may be eliminated. This gives the presentation of $\pi(D)$ defined above. 
\end{proof} 

\subsection{The proof of Theorem \ref{thm:alternating}}

Let $L$ be an alternating link. We begin by reducing the proof of the theorem to the case where $L$ is non-split.

Suppose that $L$ is split and Theorem \ref{thm:alternating} holds for non-split alternating links. The fundamental group of the $2$-fold cover of $S^3$ branched over $L$ is of the form
$$\pi_1(\Sigma(L)) \cong F_{n-1} * \pi_1(\Sigma(L_1)) * \pi_1(\Sigma(L_2)) * \ldots * \pi_1(\Sigma(L_n))$$ 
where $n \geq 2$, $F_{n-1}$ is free of rank $n - 1$, and $L_1, L_2, \ldots , L_n$ are non-split alternating links. 
By assumption, $\pi_1(\Sigma(L_j))$ is not left-orderable for each $j$. Hence if $\pi_1(\Sigma(L))$ is left-orderable, each $\pi_1(\Sigma(L_j))$ is the trivial group. It follows that each $L_j$ is a trivial knot and therefore $L$ is a trivial link of two or more components, so Theorem \ref{thm:alternating} holds.  

Assume next that $L$ is non-split and let $D$ be an alternating diagram for $L$. Label its arcs $1$ through $n$ and note that the crossings of $D$ correspond somewhat ambiguously to ordered label triples $(i, j, k)$ where $j$ is the label of the overcrossing arc. (Thus $(i, j, k)$ and $(k, j, i)$ represent the same crossing.)

Theorem \ref{thm:alternating} clearly holds when $L$ is the trivial knot so we suppose below that it isn't. Then $\pi_1(\Sigma(L))$ is non-trivial so that $\pi(D) \cong \pi_1(\Sigma(L)) * \mathbb Z$ is not abelian. Vinogradov proved that the free product of two non-trivial groups is left-orderable if and only if the two factors are left-orderable \cite{Vinogradov1949}. Thus $\pi_1(\Sigma(L))$ is left-orderable if and only if $\pi(D) \cong \pi_1(\Sigma(L)) * \mathbb Z$ is left-orderable, so the theorem will follow if we show that the hypothesis that $\pi(D)$ is left-orderable implies that $\pi(D)$ is abelian. Suppose then that ``$<$" is a left-ordering on $\pi(D)$. 

Consider the black-white checkerboard pattern on $S^2$ determined by $D$ where we assume that the black regions lie to the left as we pass over a crossing. (This convention is illustrated in Figure \ref{fig:arc-labels}.)  

Fix a crossing $(i, j, k)$. Relation (1) shows that $a_j^{-1} a_i = (a_j^{-1} a_k)^{-1}$. It follows that exactly one of the following three possibilities occurs:
$$a_i < a_j < a_k \eqno{(5)}$$ 
$$a_k < a_j < a_i \eqno{(6)}$$
$$a_i = a_j = a_k \eqno{(7)}$$
We use these options to define a semi-oriented graph $\Gamma(D)$ in $S^2$ as follows: the vertices of $\Gamma(D)$ correspond to the black regions of $D$, the edges $e$ correspond to the crossings of $D$, and the embedding in $S^2$ is that determined by $D$. Note that $\Gamma(D)$ is connected as $L$ is non-split. 

Fix an edge $e = e(i, j, k)$ and let $R_i, R_k$ be the black regions containing the arcs labelled $i, k$ respectively. Orient $e$ from $R_i$ to $R_k$ if possibility (5) occurs, from $R_k$ to $R_i$ if (6) occurs, and do not orient it if (7) occurs (see Figure \ref{fig:orientations}).  

\begin{figure}[ht!]
\begin{center}
\labellist \small
	\pinlabel $j$ at 115 195
	\pinlabel $j$ at 14 70
	\pinlabel $i$ at 14 195
	\pinlabel $k$ at 115 70
	\pinlabel $e$ at 75 51
	\pinlabel $a_i<a_j<a_k$ at 65 -10
\endlabellist
\raisebox{-3pt}{\includegraphics[scale=0.5]{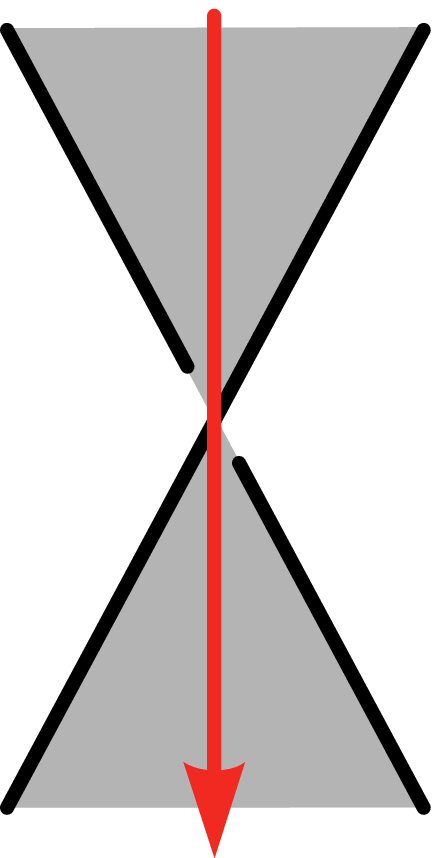}}
\qquad\qquad
\labellist \small
	\pinlabel $j$ at 115 190
	\pinlabel $j$ at 14 65
	\pinlabel $i$ at 14 190
	\pinlabel $k$ at 115 65
	\pinlabel $e$ at 75 46
	\pinlabel $a_k<a_j<a_i$ at 65 -15
\endlabellist
\raisebox{0pt}{\includegraphics[scale=0.5]{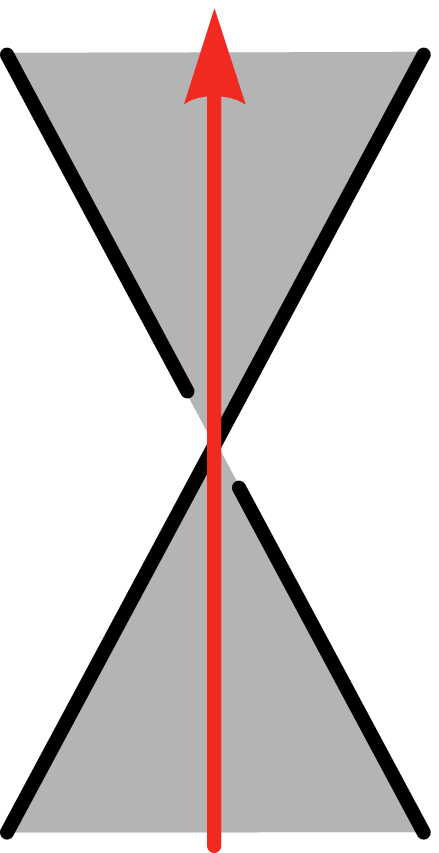}}
\qquad\qquad
\labellist \small
	\pinlabel $j$ at 115 190
	\pinlabel $j$ at 14 65
	\pinlabel $i$ at 14 190
	\pinlabel $k$ at 115 65
	\pinlabel $e$ at 75 46
	\pinlabel $a_i=a_j=a_k$ at 65 -15
	\endlabellist
\raisebox{0pt}{\includegraphics[scale=0.5]{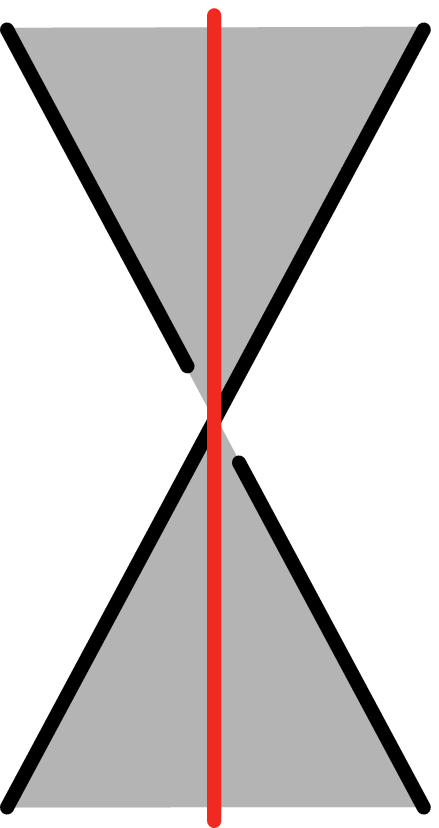}}
\end{center}
\caption{Orientation conventions for the edges of the semi-oriented graph $\Gamma(D)$.}
\label{fig:orientations}
\end{figure}

A {\it circuit} in $\Gamma(D)$ is a simple closed curve in $S^2$ determined by a sequence of edges $C: e_1, e_2, \ldots , e_m$ of $\Gamma(D)$ indexed so that successive edges are incident, and those edges of $C$ which are oriented, are oriented coherently. 

A {\it cycle} in $\Gamma(D)$ is an innermost circuit in $\Gamma(D)$. Equivalently it is a circuit which bounds a white region of $D$. 

\begin{lemma} \label{cycles unoriented} 
Each edge contained in a cycle of $\Gamma(D)$ is unoriented. 
\end{lemma}

\begin{proof} Suppose that $\Gamma(D)$ contains a cycle and let $R$ be the white region of $D$ it determines. We can label its boundary edges $i_1, i_2, \ldots , i_r$ so that the cycle is given, up to reversing its order, by the sequence of crossings $(i_1, i_2, k_1), (i_2, i_3, k_2), \ldots, (i_r, i_1, k_r)$ (see Figure \ref{fig:pentagon}). The cycle condition implies that either 
$$a_{i_1} \leq a_{i_2} \leq \ldots \leq a_{i_r} \leq a_{i_1}$$
or
$$a_{i_1} \geq a_{i_2} \geq \ldots \geq a_{i_r} \geq a_{i_1}$$
Thus all inequalities are equalities, so none of the edges of the cycle are oriented. 
\end{proof} 

A {\it sink}, respectively {\it source}, of $\Gamma(D)$ is a vertex $v$ of $\Gamma(D)$ such that each oriented edge of $\Gamma(D)$ incident to $v$ points into, respectively away from, $R$. 

\begin{lemma} \label{sinks/sources unoriented} 
Each edge incident to a source or sink in $\Gamma(D)$ is unoriented. 
\end{lemma}

\begin{proof} Let $R$ be a black region determined by $D$ and let $i_1, i_2, \ldots , i_r$ be the labels of its boundary arcs indexed (mod $r$) so that the crossings of $D$ incident to $R$ are determined by the black corners $(i_1, i_2), (i_2, i_3), \ldots , (i_r, i_1)$ (see Figure \ref{fig:pentagon}).

\begin{figure}[ht!]
\begin{center}
\labellist \small
	\pinlabel $R$ at 142 140
	\pinlabel $i_2$ at 176 189
	\pinlabel $i_1$ at 112 189
	\pinlabel $i_r$ at 80 122
	\pinlabel $\cdots$ at 142 77
	\pinlabel $i_3$ at 207 122
\endlabellist
\raisebox{0pt}{\includegraphics[scale=0.5]{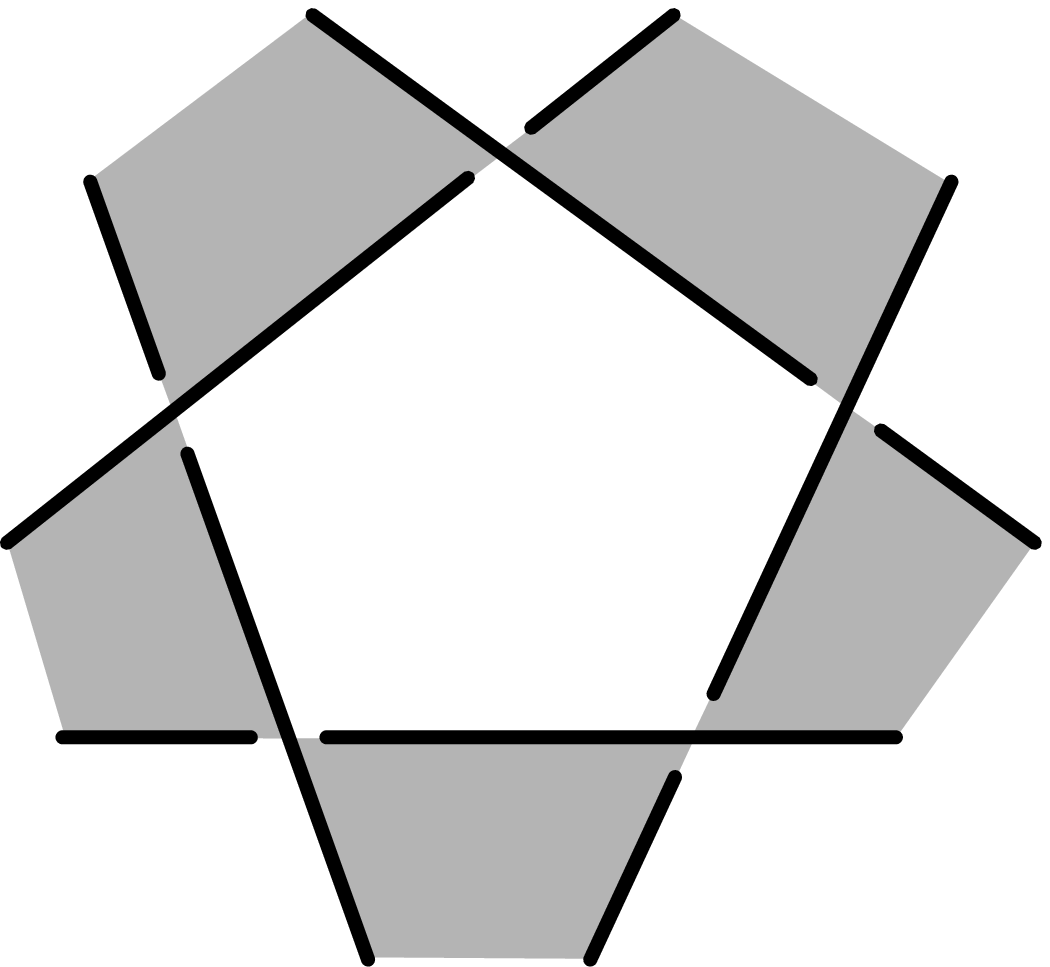}}\qquad\qquad
\labellist \small
	\pinlabel $R$ at 142 140
	\pinlabel $i_2$ at 176 189
	\pinlabel $i_1$ at 112 189
	\pinlabel $i_r$ at 80 122
	\pinlabel $\cdots$ at 142 77
	\pinlabel $i_3$ at 207 122
\endlabellist
\raisebox{0pt}{\includegraphics[scale=0.5]{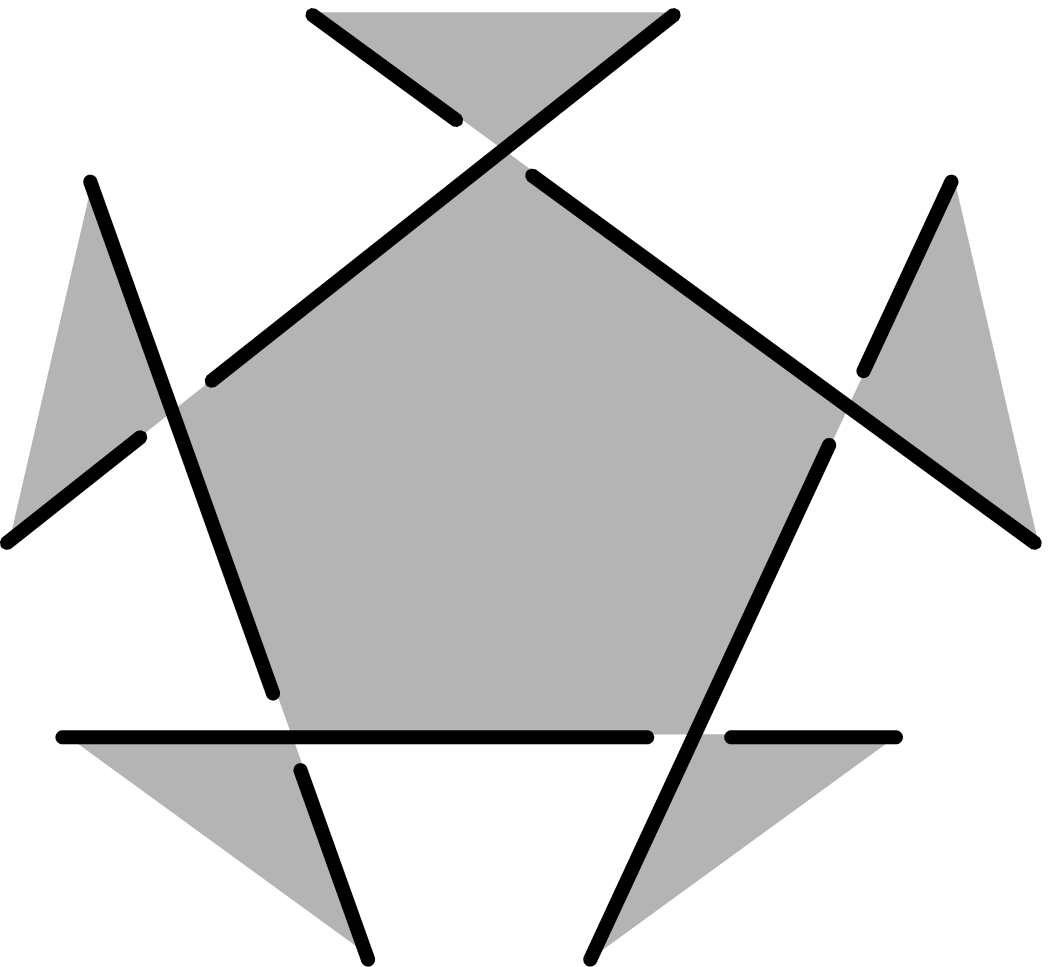}}
\end{center}
\caption{White and black regions of the checkerboard colouring considered in the proofs of Lemma \ref{cycles unoriented} and Lemma \ref{sinks/sources unoriented}, respecively. }
\label{fig:pentagon}
\end{figure}

Then $a_{i_1} \leq a_{i_2} \leq \ldots \leq a_{i_r} \leq a_{i_1}$ if the vertex $v$ is a sink, while $a_{i_1} \geq a_{i_2} \geq \ldots \geq a_{i_r} \geq a_{i_1}$ if it is a source. In either case, $a_{i_1} = a_{i_2} = \ldots = a_{i_r}$, so no edge incident to $v$ is oriented. 
\end{proof} 

On the other hand, we have the following result. 

\begin{lemma} \label{some oriented}
Let $\Gamma$ be a connected semi-oriented graph in $S^2$ without sinks or sources containing oriented edges. If some edge $e$ of $\Gamma$ is oriented, then there is a cycle of $\Gamma$ containing an oriented edge.
\end{lemma}

\begin{proof} Since $\Gamma$ has no sinks, the vertex at the head of an oriented edge is incident to the tail of another oriented edge of $\Gamma$. Starting from some oriented edge, we obtain a first return circuit, all of whose edges are oriented.

Choose a circuit $C$ in $\Gamma$ which is innermost among the family of circuits all of whose edges are oriented. Then $C$ bounds a disk $E$ such that each circuit $C' \ne C$ contained in $\Gamma \cap E$ has an unoriented edge. Suppose some edge $e$ of $(\Gamma\cap E) \setminus C$ is oriented. Arguing as in the first paragraph of this proof, and using our innermost assumption on $C$, we obtain an oriented path of edges starting with $e$ and ending at some vertex in $C$. Similarly, since $\Gamma$ has no sources, we can carry out the same construction backwards, starting from the tail of $e$. This produces a (non-empty) oriented path of edges in $(\Gamma \cap E) \setminus C$ starting and ending on $C$, which in turn gives a circuit of oriented edges that contradicts the innermost property of $C$. Thus each edge of $\Gamma \setminus C$ contained in $E$ is unoriented. It follows that the boundary of any region of $D$ contained in $E$ whose boundary contains an edge of $C$ determines a cycle containing an oriented edge. 
\end{proof} 

The last three lemmas combine to show that each edge of $\Gamma(D)$ is unoriented. It follows that for each crossing 
$(i, j, k)$, 
$$a_i = a_j = a_k \eqno{(8)}$$ 
Hence if $R$ is a black region of $D$ with boundary arcs labelled successively $i_1, i_2, \ldots , i_r$  (see Figure 4), then $a_{i_1} = a_{i_2} = \ldots = a_{i_r}$. Equation (8) shows that the $a_l$ determined by boundary arcs of black regions sharing a corner are the same. Since any two black regions are connected by a chain of black regions for which successive regions share a corner, it follows that $a_1 = a_2 = \ldots = a_n$. Hence $\pi(D) \cong \pi_1(\Sigma(L)) * \mathbb Z$ is abelian. As we noted above, this implies Theorem \ref{thm:alternating}. 
\qed  

\section{Surgeries on alternating knots} \label{sec:fig 8}

In this section we prove Propositions \ref{prop:alt} and \ref{prop:f8}. These results provide examples, many hyperbolic, of non-L-spaces with left-orderable fundamental groups. 

\begin{proof}[Proof of Proposition \ref{prop:alt}]  
Assertion (1) of this proposition is an immediate consequence of \cite[Theorem 1.5]{OSz2005-lens} and Proposition \ref{thm:sf}.  

Roberts has shown \cite{Roberts1995} that $S^3_{1/q}(K)$ admits a taut foliation under the hypotheses of assertions (2) and (3) of this proposition. We claim that it is also atoroidal. Menasco \cite[Corollary 1]{Menasco1982} has shown that prime alternating knots are atoroidal. On the other hand, Patton \cite{Patton1995} has shown that an alternating knot which admits an essential punctured torus is either a two-bridge knot or a three-tangle Montesinos knot. As the boundary slopes of these types of alternating knots are even integers \cite{HT1985}, \cite{HO1989}, $S^3_{1/q}(K)$ is atoroidal under the hypotheses of assertions (2) and (3). As $H_1(S^3_{1/q}(K)) = 0$, the foliation is co-oriented and the representation $\pi_1(S^3_{1/q}(K)) \to \hbox{Homeo}_+(S^1)$ provided by Thurston's universal circle construction \cite{CD2003} lifts to $\widetilde{\hbox{Homeo}}_+(S^1)$. As this representation is injective, \cite[Theorem 1.1 (1)]{BRW2005} implies that $\pi_1(S^3_{1/q}(K))$ is left-orderable. This completes the proof. 
\end{proof} 

Now we proceed to the proof of Proposition \ref{prop:f8}.
Our argument is based on a result of Khoi stated on page 795 of \cite{Khoi2003} and justified through a reference to a MAPLE calculation, though no details are given. Because of the importance of his result to our treatment, we provide a proof of it below. See Proposition \ref{section}.

Let $M$ denote the exterior of the figure eight knot. We know from \cite[Example 3.13]{BRW2005} that there are a continuous family of representations with non-abelian image \[\rho_s : \pi_1(M) \to PSL_2(\bR), \; s \geq \textstyle\frac{1 + \sqrt{5}}{2}\] and a continuous function \[g: \left[\textstyle\frac{1 + \sqrt{5}}{2}, \infty\right) \to [0, \infty)\] such that $g(s) = r \in \bQ$ if and only if $\rho_{s}(\mu^p \lambda^q) = \pm I$ where $r = \pq$ is a reduced fraction. Further, the image of $g$ contains $[0, 4)$. 

Consider the universal covering homomorphism $\psi: \widetilde{SL_2} \to PSL_2(\bR)$. The kernel $\mathcal{K}$ of $\psi$ is the centre of $\widetilde{SL_2}$ and is isomorphic to $\mathbb Z$. There is a lift of $\rho_s$ to a homomorphism $\tilde \rho_s: \pi_1(M) \to \widetilde{SL_2}$ since the obstruction to its existence is the Euler class $e(\rho_s) \in H^2(\pi_1(M); \mathcal{K}) \cong H^2(M; \mathbb Z) \cong 0$ \cite[Section 6.2]{Ghys2001}. The set of all such lifts is a transitive $H^1(\pi_1(M); \mathcal{K})$ set where for $\phi \in H^1(\pi_1(M); \mathcal{K})$   
\[(\phi \cdot \widetilde \rho_s)(\gamma) = \phi(\gamma)\widetilde \rho_s(\gamma)\] 

There is an identification $\widetilde{SL_2} \cong \Delta \times \bR$ where $\Delta = \{z \in \mathbb C : |z| < 1\}$ in which the following properties hold:
 \begin{itemize}

\item the identity is represented by $(0,0)$;

\vspace{.3cm} \item $\mathcal{K}$ corresponds to $\{(0, k \pi) : k \in \mathbb Z\}$ 

\vspace{.3cm} \item if the image of $(z, \omega)$ in $SL_2(\mathbb R)$ has positive eigenvalues, then there is an even integer $2j$ such that $|\omega - 2j \pi| < \frac{\pi}{2}$. Further, $(z, \omega)$ is conjugate to an element of the form $(r, 2j \pi)$ where $r \in (-1,1)$;  

\vspace{.3cm} \item for $r \in (-1,1)$, the centralizer of $(r, 0)$ is contained in $\bigcup_{j \in \mathbb Z} (-1,1) \times \{j \pi\}$.

\end{itemize}
See \cite[Section 2]{Khoi2003}, for instance, for the details. 

The action of $PSL_2(\mathbb R)$ on the circle induces an inclusion $\widetilde{SL_2} \leq \widetilde{\hbox{Homeo}}_+(S^1) = \{f \in \hbox{Homeo}_+(\mathbb R) : f(x+1) = f(x) + 1 \hbox{ for all } x \in \mathbb R\}$. In this case, 
 \begin{itemize}

\item if $(z, \omega)$ is a commutator, then $-\frac{3\pi}{2} < \omega < \frac{3\pi}{2}$  \cite[Inequality 4.4 and Proposition 4.8]{Wood}

\end{itemize}

\begin{proposition} \label{section} {\rm (Khoi)}
Let $\tilde \rho: \pi_1(M) \to \widetilde{SL_2}$ be a homomorphism. Then up to conjugation and replacing $\tilde \rho$ by a representation $\tilde \rho' = \phi \cdot \widetilde \rho$ for some $\phi \in H^1(\pi_1(M); \mathcal{K})$, we can suppose that $\tilde \rho(\pi_1(\partial M))$ is contained in the $1$-parameter subgroup $(-1, 1) \times \{0\}$ of $\widetilde{SL_2}$. 
\end{proposition}

\begin{proof}[Proof of Proposition \ref{prop:f8}] 
Let $r = \pq \in [0,4)$ and fix $s$ so that $g(s) = r$. Then $\widetilde \rho_{s}(\mu^p \lambda^q) \in \mathcal{K} = \{(0, k\pi) : k \in \mathbb Z\}$. On the other hand, by Proposition \ref{section} we may assume $\tilde \rho_s(\pi_1(\partial M)) \subset (-1, 1) \times \{0\}$. Thus $\widetilde \rho_{s}(\mu^p \lambda^q) = (0,0) = 1$. It follows that $\tilde \rho_s$ induces a homomorphism $\pi_1(S^3_r(K)) \to \widetilde{SL_2}$ with non-abelian image. Since $S_r^3(K)$ is irreducible for all $r$ \cite[Theorem 2(a)]{HT1985} and $\widetilde{SL_2}$ is left-orderable \cite{Conrad1959}, \cite[Theorem 1.1]{BRW2005} implies that $\pi_1(S^3_r(K))$ is left-orderable for all $r \in [0, 4)$. Since $K$ is amphicheiral, $\pi_1(S^3_{r}(K))$ is left-orderable for $r \in (-4,0]$, so we are done.
\end{proof} 

The rest of this section is devoted to the proof of Proposition \ref{section}. First we develop some background material. 

Consider the presentation 
$$\pi_1(M) = \langle x, y, t : txt^{-1} = xyx^2, tyt^{-1} = x^{-1}\rangle$$ 
Here $t$ is a meridional class and $\lambda = [x,y]$ a longitudinal class. The reader will verify that $[t, \lambda] = 1$. 
Set $F = \langle x, y \rangle \lhd  \pi$.  

We denote the $SL_2(\mathbb C), SL_2(\mathbb R)$, and $SU(2)$ character varieties of a group $\Gamma$ by $X_{SL_2(\mathbb C)}(\Gamma)$, $X_{SL_2(\mathbb R)}(\Gamma)$, and $X_{SU(2)}(\Gamma)$ respectively. The character of a representation $\rho$ will be denoted by $\chi_\rho$. 

There is a homeomorphism \cite[Proposition 4.1]{Goldman1988} 
$$\Psi: X_{SL_2(\mathbb C)}(F) \cong \mathbb C^3, \chi \mapsto (\chi(x), \chi(y), \chi(xy))$$
It is known that 
$$\Psi^{-1}(\mathbb R^3) = \{\chi_\rho : \rho \hbox{ has image in } SL_2(\mathbb R) \hbox{ or } SU(2)\}$$ 
See \cite[Proposition III.1.1]{MS1984}. 

\begin{lemma} \label{image}
The image of the composition of $\Psi$ with the restriction induced map $X_{SL_2(\mathbb C)}(\pi) \to X_{SL_2(\mathbb C)}(F)$ is $X_0 = \{(a , a, \frac{a}{a-1}) \in \mathbb C^3 : a \ne 1\}$. 
\end{lemma}

\begin{proof} The identity $\hbox{trace}(AB) + \hbox{trace}(AB^{-1}) = \hbox{trace}(A) \hbox{trace}(B)$ for $A, B \in SL_2(\mathbb C)$ implies that for each $\chi \in X_{SL_2(\mathbb C)}(\pi_1(M))$ and $z, w \in \pi_1(M)$ we have 
$$\chi(zw) + \chi(zw^{-1}) = \chi(z) \chi(w) \eqno{(9)}$$
Given such a $\chi$ set $a = \chi(x)$. The relation $x^{-1} = t y t^{-1}$ implies that $\chi(y) = a$. Next note that $\chi(xy) = \chi(txyt^{-1}) = \chi(xyx) = \chi(xyx) + \chi(xyx^{-1}) - \chi(xyx^{-1}) = \chi(xy)\chi(x) - \chi(y) = a\chi(xy) - a$. Thus $\chi(xy) = \frac{a}{a-1}$ so the image of the composition of $\Psi$ with the restriction induced map $X_{SL_2(\mathbb C)}(\pi) \to X_{SL_2(\mathbb C)}(F)$ is contained in $X_0$.  

Conversely fix $(a,a, \frac{a}{a-1}) \in X_0$ and consider the isomorphism $\theta: F \to F$ given by $x \mapsto xyx^2, y \mapsto x^{-1}$. There is a semisimple representation $\rho_0: \pi_1(F) \to SL_2(\mathbb C)$ such that $\Psi(\chi_{\rho_0}) =  (a , a, \frac{a}{a-1})$. Let $\rho_1 = \rho_0 \circ \theta$. It is easy to see that $\rho_0$ and $\rho_1$ have the same character. Since they are semisimple there is an $A \in SL_2(\mathbb C)$ such that $\rho_1 = A \rho_0 A^{-1}$. It is easy to see then that there is a representation $\rho: \pi \to SL_2(\mathbb C)$ such that $\rho(t) = A$ and $\rho|\pi_1(F) = \rho_0$. Hence $(a , a, \frac{a}{a-1})$ lies in the image of image of the composition of $\Psi$ with the restriction induced map $X_{SL_2(\mathbb C)}(\pi) \to X_{SL_2(\mathbb C)}(F)$, which completes the proof of the lemma.  
\end{proof}

Let 
$\kappa: \mathbb C^3 \to \mathbb C$ be given by
$$\kappa(a,b,c) = a^2 + b^2 + c^2 - abc - 2$$  
Then for $\chi \in X_{SL_2(\mathbb C)}(F)$ and $(\chi(x), \chi(y), \chi(xy)) = (a,b,c)$, Identity (9) implies that $\chi([x,y])) = \kappa(a,b,c)$. 

\begin{proposition} \label{su2} {\rm \cite[Theorem 4.3]{Goldman1988}}
Let $\chi \in X_{SL_2(\mathbb C)}(F)$ and set $(a,b,c) = (\chi(x), \chi(y), \chi(xy))$. Then $\chi \in X_{SU(2)}(F)$ if and only if $(a,b, \kappa(a,b,c)) \in [-2, 2]^3$.  
\qed
\end{proposition}

A straightforward calculation shows that for $a \in \mathbb R \setminus \{1\}$, 
\begin{itemize}
\vspace{-.1cm} \item $\kappa(a , a, \frac{a}{a-1}) = (a-1)^2 + (a-1) - 2 + (a-1)^{-1} + (a-1)^{-2} \geq -2$, and 

\vspace{.2cm} \item $(a,a, \kappa(a,a, \frac{a}{a-1})) \in [-2, 2]^3$ if and only if $-\big(\frac{\sqrt{5} + 1}{2}\big) \leq a \leq \frac{\sqrt{5} - 1}{2}$ or $a = 2$.
\end{itemize}

\begin{lemma} \label{sl2}
The image of the composition of $\Psi$ with the restriction map $X_{SL_2(\mathbb R)}(\pi_1(M)) \to X_{SL_2(\mathbb R)}(F)$ is $Y_0 = \{(a,a,\frac{a}{a-1}) : a \in \mathbb R \setminus \{1\}, \kappa(a,a,\frac{a}{a-1}) \geq 2\}$. In particular, if $\chi_\rho \in X_{SL_2(\mathbb R)}(\pi_1(M))$, the eigenvalues of $\rho(\lambda)$ are positive reals. 
\end{lemma}

\begin{proof} 
Fix $\chi_\rho \in X_{SL_2(\mathbb R)}(\pi_1(M))$. If $\rho$ is reducible, then $\chi(\pi_1(F)) = \{2\}$ since $\pi_1(F)$ is contained in the commutator subgroup of $\pi_1(M)$. Thus $\Psi(\chi_\rho) = (2,2,2) \in Y_0$. Suppose then that $\chi_\rho$ is irreducible. The image of $\rho$ leaves a geodesic plane $P$ in $\mathbb H^3$ invariant so it cannot conjugate into $SU(2)$; otherwise it would fix a point of $P$ and therefore be conjugate into $SO(2)$ contrary to the irreduciblity of $\rho$. Thus $\Psi(\chi_\rho) \not \in X_{SU(2)}(\pi_1(M))$. It follows from Proposition \ref{su2} that $\Psi(\chi_\rho) \in \{(a,a,\frac{a}{a-1}) : a \in \mathbb R \setminus \{1\}, \kappa(a,a,\frac{a}{a-1}) \not \in [-2, 2]\}$. Since $\kappa(a , a, \frac{a}{a-1}) \geq -2$ for all $a \in \mathbb R \setminus \{1\}$, $\Psi(\chi_\rho) \in Y_0$. 

Finally observe that if $\chi_\rho \in X_{SL_2(\mathbb R)}(\pi_1(M))$ and $\rho(\lambda)$ has eigenvalues $\zeta, \zeta^{-1} \in \mathbb C^*$, then $\zeta + \zeta^{-1} = \hbox{trace}(\rho(\lambda)) = \kappa(a,a, \frac{a}{a-1}) \geq 2$. Thus $\zeta$ is a positive real number. 
\end{proof}

\begin{proof}[Proof of Proposition \ref{section}]
The properties of $\widetilde{SL_2} = \Delta \times \mathbb R$ listed just before the statement of Proposition \ref{section} will be used without direct reference in the proof. 

Since $\lambda$ is a commutator, if $\tilde \rho(\lambda) = (z, \omega)$ then $-\frac{3\pi}{2} < \omega < \frac{3\pi}{2}$. On the other hand, since the eigenvalues of $\psi (\tilde \rho(\lambda))$ are positive there is an even integer $2k$ such that $|\omega - 2k\pi| < \frac{\pi}{2}$. Hence $-\frac{\pi}{2} < \omega < \frac{\pi}{2}$ and therefore $\tilde \rho(\lambda)$ is conjugate into the subgroup $(-1,1) \times \{0\}$ of $\widetilde{SL_2}$. Without loss of generality we assume $\tilde \rho(\lambda) \in (-1,1) \times \{0\}$. Since $\tilde \rho(t)$ commutes with $\tilde \rho(\lambda)$, there is an integer $j$ such that $\tilde \rho(t) \in (-1,1) \times \{j\pi\}$. Fix $\phi_0 \in H^1(M; \mathcal{K}) = \hbox{Hom}(\pi_1(M), \mathcal{K})$ such that $\phi_0(t) = (0, -j\pi)$ and set $\tilde \rho' = \phi \cdot \widetilde \rho$. From the multiplication on $\widetilde{SL_2}$ (cf. \cite[page 764]{Khoi2003}) we see that $\tilde \rho'(t) \subset (-1,1) \times \{0\}$. Then $\tilde \rho'(\pi_1(\partial M)) \subset (-1,1) \times \{0\}$, 
which completes the proof. 
\end{proof}

\section{Left-orderability and representations with values in $\hbox{Homeo}_+(S^1)$} \label{sec: square trivial}

In this section we prove Theorem \ref{thm:square trivial}, Corollary \ref{cor:alt square-trivial} and Corollary \ref{cor:trace-field}. 

\begin{proof}[Proof of Theorem \ref{thm:square trivial}] Let $K$ be a prime knot in the $3$-sphere and suppose that $\pi_1(\Sigma(K))$ is not left-orderable. Let $M_K$ denote the exterior of $K$ and fix a homomorphism $\rho: \pi_1(M_K) \to \hbox{Homeo}_+(S^1)$ such that $\rho(\mu^2) = 1$ for each meridional class $\mu$ in $\pi_1(M_K)$. We will show that the image of $\rho$ is either trivial or isomorphic to $\mathbb Z/2$.  

Let $p: \widetilde M_K \to M_K$ be the $2$-fold cover determined by the epimorphism $\pi_1(M_K) \to \mathbb Z/2$. Then $\Sigma(K)$ is obtained by filling the boundary component of $\widetilde M_K$ along the inverse image of a meridional curve of $M_K$. 

The Euler class $e(\rho)$ of $\rho$ (\cite[Section 6.2]{Ghys2001}) is contained in $H^2(M_K; \mathbb Z) \cong 0$, and so is zero. Hence if $\widetilde \rho = \rho| \pi_1(\widetilde M_K)$, then $e(\widetilde \rho) = p^*(e(\rho)) = 0$. Our assumptions imply that $\widetilde \rho$ induces a homomorphism $\psi: \pi_1(\Sigma(K)) \to \hbox{Homeo}_+(S^1)$ such that if $i: \widetilde M_K \to \Sigma(K)$ is the inclusion, $e(\widetilde \rho) = i^*(e(\psi))$. 

Let $\widehat K = p^{-1}(K) \subset \Sigma(K)$ and let $N(\widehat K)$ denote a closed tubular neighbourhood of $\widehat K$ in $\Sigma(K)$. Note that  $H_1(\widetilde M_K; \mathbb Z) \cong H_1(\Sigma(K); \mathbb Z) \oplus H_2(\Sigma(K), \widetilde M_K; \mathbb Z) \cong H_1(\Sigma(K); \mathbb Z) \oplus \mathbb Z$ where the $\mathbb Z$ factor is generated by the boundary of a meridian disk of $N(\widehat K)$. It follows that the connecting homomorphism $H^1(\widetilde M_K; \mathbb Z) = \mbox{Hom}(H_1(\widetilde M_K; \mathbb Z), \mathbb Z) \to H^2(\Sigma(K), \widetilde M_K; \mathbb Z) = \mbox{Hom}(H_2(N(\widehat K), \partial N(\widehat K); \mathbb Z); \mathbb Z)$ is surjective. Thus $i^*: H^2(\Sigma(K); \mathbb Z) \to H^2(\widetilde M_K; \mathbb Z)$ is injective. Then as $i^*(e(\psi)) = e(\widetilde \rho) = 0$, $e(\psi) = 0$. In particular, $\psi$ lifts to a homomorphism $\widetilde \psi : \pi_1(\Sigma(K)) \to \widetilde{\hbox{Homeo}}_+(S^1) \leq \hbox{Homeo}_+(\mathbb R)$ \cite[Section 6.2]{Ghys2001}. Since $K$ is prime, $\Sigma(K)$ is irreducible. Further, $\hbox{Homeo}_+(\mathbb R)$ is left-orderable \cite[Theorem 7.1.2]{MR1977}, and therefore as $\pi_1(\Sigma(K))$ is not left-orderable, $\widetilde \psi$ is the trivial homomorphism \cite[Theorem 1.1]{BRW2005}. The same conclusion then holds for $\widetilde \rho$ and hence the image of $\rho$ is a cyclic group of order dividing $2$. 
\end{proof} 

\begin{proof}[Proof of Corollary \ref{cor:alt square-trivial}] Let $K$ be an alternating knot and $\rho: \pi_1(S^3 \setminus K) \to \hbox{Homeo}_+(S^1)$ a homomorphism such that $\rho(\mu^2) = 1$ for each meridional class $\mu$ in $\pi_1(S^3 \setminus K)$. Corollary \ref{cor:alt square-trivial} clearly holds when $K$ is trivial, so suppose it isn't and let $K_1, K_2, \ldots, K_n$ be its prime factors. Each $K_i$ is alternating and 
$$\pi_1(S^3 \setminus K) \cong \pi_1(S^3 \setminus K_1) *_{\mu_1 = \mu_2} \pi_1(S^3 \setminus K_2) *_{\mu_2' = \mu_3} \ldots *_{\mu_{n-1}' = \mu_n} \pi_1(S^3 \setminus K_n)$$
where $\mu_i, \mu_i'$ are meridional classes of $K_i$. Further, a meridional class of $K_i$ is a meridional class of $K$. Hence Theorem \ref{thm:square trivial} implies that for each $i$, $\rho(\pi_1(S^3 \setminus K_i))$ is a subgroup of $\mathbb Z/2$. Then $\rho|\pi_1(S^3 \setminus K_i)$ factors through $H_1(S^3 \setminus K_i)$ and therefore $\rho(\pi_1(S^3 \setminus K_i))$ is generated by $\rho(\mu_i)$ and $\rho(\mu_i) = \rho(\mu_i')$ for $2 \leq i \leq n-1$. Given our presentation for $\pi_1(S^3 \setminus K)$, Corollary \ref{cor:alt square-trivial} is a straightforward consequence of these observations. 
\end{proof} 

\begin{proof}[Proof of Corollary \ref{cor:trace-field}] 

Finally, consider the hypotheses of Corollary \ref{cor:trace-field} and let $M_K$ be the exterior of $K$.  If the trace field of $\pi_1(\mathcal{O}_K(2))$ has a real
embedding, then it determines an irreducible representation of $\pi_1(M_K)$ which conjugates into $PSL_2(\mathbb R)$ or $PSU(2)$ \cite[Definition 7.2.1]{MR2003}. The former is ruled out by Corollary \ref{cor:alt square-trivial}. Thus Corollary \ref{cor:trace-field} holds. 
\end{proof}

\bibliographystyle{plain}
\bibliography{BGW}

\end{document}